\documentclass{article}       
\usepackage{graphicx}
\usepackage{todonotes}
\usepackage[english]{babel}
\usepackage{enumitem}
\usepackage{amsthm, amsmath,amsfonts,amssymb,mathtools}
\usepackage{trfsigns}
\usepackage{units}
\usepackage{hyperref}
\newtheorem{thm}{Theorem}
\numberwithin{thm}{section}
\newtheorem{lem}[thm]{Lemma}
\newtheorem{cor}[thm]{Corollary}
\newtheorem{prop}[thm]{Proposition}
\newtheorem{exam}[thm]{Example}
\newtheorem*{acknowledgements}{Acknowledgements}

\newtheorem{defi}[thm]{Definition}
\newtheorem{assumption}[thm]{Assumption}

\newcommand{\new}{}
\newcommand{\newnew}{}
\newcommand{\pstar}{\tilde{g}}
\newcommand{\pconj}{p^*}

\newtheorem{rem}[thm]{Remark}

\newcommand{\N}{\mathbb{N}}
\newcommand{\Z}{\mathbb{Z}}
\newcommand{\R}{\mathbb{R}}

\newcommand{\torus}{\mathbb{T}}
\newcommand{\X}{\mathcal{X}}
\newcommand{\Y}{\mathcal{Y}}
\newcommand{\HilbertX}{\X}
\newcommand{\HilbertY}{\Y}
\newcommand{\manifold}{\mathcal{M}}
\newcommand{\norm}[1]{\left\|#1\right\|}
\newcommand{\paren}[1]{\left(#1\right)}
\newcommand{\lsp}{\left\langle}
\newcommand{\rsp}{\right\rangle}

\newcommand*\diff{\mathop{}\!\mathrm{d}}

\newcommand{\T}{T} 
\newcommand{\f}{f} 
\newcommand{\g}{g}
\newcommand{\p}{p}	
\newcommand{\fal}{\hat{f}_\alpha}
\newcommand{\fdagger}{f^\dagger}
\newcommand{\gdagger}{g^\dagger}
\newcommand{\itfal}[1]{\hat{f}_\alpha^{(#1)}}
\newcommand{\pal}{\hat{p}_\alpha}
\newcommand{\itpal}[1]{\hat{p}_\alpha^{(#1)}}
\newcommand{\pbar}{\overline{p}}
\newcommand{\wbar}{\overline{\omega}}
\newcommand{\itpbar}[1]{\overline{p}^{(#1)}}
\newcommand{\itwbar}[1]{\overline{\omega}^{(#1)}}
\newcommand{\gobs}{g^{\mathrm{obs}}}
\newcommand{\Sfun}{\mathcal{S}} 
\newcommand{\Rpen}{\mathcal{R}}
\newcommand{\Ritpen}[1]{\mathcal{R}_{#1}}
\newcommand{\q}{q} 
\newcommand{\conv}{r} 
\newcommand{\smoo}{s} 
\newcommand{\cXR}{c_{q,\Y}} 
\newcommand{\csXR}{c_{q^*\!,\Y^*}} 
\newcommand{\breg}[1]{\Delta_{#1}}
\newcommand{\symbreg}[1]{\Delta_{#1}^{\rm{sym}}}
\newcommand{\landauO}[1]{\mathcal{O}\left(#1\right)}

\renewcommand{\qed}{}
\newcommand{\keywords}[1]{\ \\ {\bf Keywords:} #1 \\ {\bf MSC:} 65J20, 65J22  }

\DeclareMathOperator{\sgn}{sgn}
\DeclareMathOperator*{\argmin}{arg\,min}
\DeclareMathOperator{\KL}{KL}

\DeclareMathOperator{\range}{ran}
\DeclareMathOperator{\VSC}{VSC}
\DeclareMathOperator{\id}{id}
\DeclareMathOperator{\dom}{dom}
\DeclareMathOperator{\supp}{supp}

\title{Higher order convergence rates for Bregman iterated variational regularization 
of inverse problems}

\author{Benjamin Sprung\thanks{Lotzestr. 16-18, 37083 G\"ottingen, Germany, \href{mailto:b.sprung@math.uni-goettingen.de}{b.sprung@math.uni-goettingen.de},
 \href{mailto:hohage@math.uni-goettingen.de}{hohage@math.uni-goettingen.de}}    \and      Thorsten Hohage\footnotemark[1] 
}

\begin{document}

\maketitle

\begin{abstract}
We study the convergence of variationally regularized solutions to linear ill-posed operator 
equations in Banach spaces as the noise in the right hand side tends to $0$. The rate of 
this convergence is determined by abstract smoothness conditions on the solution called 
source conditions. For non-quadratic data fidelity or penalty terms such source conditions are 
often formulated in the form of variational inequalities. Such variational source conditions (VSCs) 
as well as other formulations of such conditions in Banach spaces have the disadvantage 
of yielding only low-order convergence rates. A first step towards higher order 
VSCs has been taken by Grasmair (2013) who obtained convergence 
rates up to the saturation of Tikhonov regularization. For even higher order convergence rates, 
iterated versions of variational regularization have to be considered. In this paper we 
introduce VSCs of arbitrarily high order which lead to optimal convergence 
rates in Hilbert spaces. For Bregman iterated variational regularization in Banach spaces with general data fidelity 
and penalty terms, we derive convergence rates under third order VSC. 
These results are further discussed for entropy regularization with elliptic pseudodifferential 
operators where the VSCs are interpreted in terms of Besov spaces and the optimality of the rates 
can be demonstrated. Our theoretical results are confirmed in numerical experiments. 

\end{abstract}

\section{Introduction}
We consider linear, ill-posed inverse problems in the form of operator equations 
\begin{equation}\label{eq:opeq}
\T\f=\gobs
\end{equation}
with a bounded linear operator $\T:\X\to \Y$ between Banach spaces $\X$ and $\Y$. 
We will assume that $\T$ is injective, but that $\T^{-1}:\T(\X)\to \X$ is not continuous. 
The exact solution will be denoted by $\fdagger$, and the noisy observed data by $\gobs\in\Y$, 
assuming the standard deterministic noise model 
\begin{align}\label{eq:det_noise}
\|\gobs-\T\fdagger\|\leq \delta 
\end{align}
with noise level $\delta>0$. To obtain a stable estimator of $\fdagger$ from such data 
we will consider generalized Tikhonov regularization of the form 
\begin{align}\label{eq:genTikh} 
	\fal\in\argmin_{\f\in\X}\left[\frac{1}{\alpha}\Sfun(\T\f-\gobs)+\Rpen(\f)\right] \tag{{\new $P_1$}}
\end{align}
with a convex, lower semi-continuous penalty functional $\Rpen:\X\to \R\cup\{\infty\}$, 
$\R\not\equiv \infty$, a regularization parameter $\alpha>0$ and a data fidelity term
\begin{align}\label{eq:defiS}
\Sfun(\g):=\tfrac{1}{q}\|\g\|_{\Y}^q
\end{align} 
for some $q> 1$. 

The aim of regularization theory is to bound the reconstruction error 
$\|\fal-\fdagger\|$ in terms of the noise level $\delta$. Classically, in 
a Hilbert space setting, conditions implying such bounds have been formulated in terms of the 
spectral calculus of the operator $\T^*\T$,
\begin{align}\label{eq:ssc}
\fdagger\in (\T^*\T)^{\nu/2}(\X)
\end{align}
for some $\nu>0$ and some initial guess $\f_0$. In fact, 
using spectral theory it is easy to show that \eqref{eq:ssc} yields 
\begin{align}\label{classical_opt_order_rates}
	\norm{\fdagger-\fal}= \mathcal{O}\paren{\delta^\frac{\nu}{\nu+1}} 
\end{align}
for classical Tikhonov regularization (i.e.\ \eqref{eq:genTikh} with 
$\Rpen(\f)=\|\f-\f_0\|^2$ and $\Sfun(\g)=\|\g\|_{\Y}^2$) if $\nu\in (0,2]$ 
and $\alpha\sim \delta^\frac{2}{\nu+1}$ (see e.g.\ \cite{EHN:96}). 
The proof and even the formulation 
of the source condition \eqref{eq:ssc} rely on spectral theory and have no 
straightforward generalizations to Banach space settings, and even 
in a Hilbert space setting the proof does not apply to frequently used 
nonquadratic functionals $\Rpen$ and $\Sfun$. 

As an alternative, starting from \cite{HKPS:07}, source conditions in the form 
of variational inequalities have been used: 
\begin{align}\label{eq:vsc1}
\forall \f\in \X\,:\,\lsp \f^*,\fdagger-\f\rsp 
\leq \tfrac{1}{2}\breg{\Rpen}^{\f^*}(\f,\fdagger) + \Phi\paren{\Sfun(\T\f-\T\fdagger)}
\end{align}
Here  $\Phi:[0,\infty)\to [0,\infty)$ is an index function (i.e.\ $\Phi$ 
is continuous and increasing with $\Phi(0)=0$), and $\breg{\Rpen}^{\f^*}$ denotes the 
Bregman distance (see Section \ref{sec:bregman_it} for a definition). 
Under the noise model \eqref{eq:det_noise} the variational source condition \eqref{eq:vsc1} 
implies the convergence rate 
$\breg{\Rpen}^{\f^*}(\fal,\fdagger)\leq \mathcal{O}(\Phi(\delta^2))$, as shown in \cite{grasmair:10}. 
In contrast to spectral source conditions, the condition \eqref{eq:vsc1} is not only 
sufficient, but even necessary for this rate of convergence in most cases (see \cite{HW:17}). 
Moreover, due to the close connection to conditional stability estimates, variational 
source conditions can be verified even for interesting nonlinear inverse problems \cite{HW:15}.

However, it is easy to see that \eqref{eq:vsc1} with quadratic $\Rpen$ and $\Sfun$
can only hold true for $\Phi$ satisfying
 $\lim_{\tau\to 0} \Phi(\tau)/\sqrt{\tau}>0$ (except 
for the very special case $\fdagger\in\argmin \Rpen$), see \cite[Prop.~12.10]{flemming:12b}. 
This implies that for 
quadratic Tikhonov regularization the condition \eqref{eq:vsc1} only covers spectral 
H\"older source condition \eqref{eq:ssc} with indices $\nu\in (0,1]$. 
{\new Several alternatives to the formulation \eqref{eq:vsc1} of the source condition suffer from 
the same limitation: multiplicative variational source conditions \cite{Andreev2015,KH:10}, approximate 
source conditions \cite{flemming:12b}, and approximate variational source conditions \cite{flemming:12b}. 
Symmetrized version of multiplicative variational source conditions (see \cite[eq.~(6)]{Andreev2015} and
\cite[\S 4]{Albani2016}) cover a larger range of $\nu$, but have no obvious generalization to Banach space settings 
or non-quadratic $\Sfun$ or $\Rpen$.}
As shown in the first paper \cite{HKPS:07}, the limiting case 
$\Phi(\tau)= c\sqrt{\tau}$ is equivalent to the source condition
\begin{equation}\label{eq:SSChalf}
\exists \pbar\in \Y^*\colon\qquad T^*\pbar = \f^*
\end{equation}
studied earlier in \cite{BO:04,eggermont:93}. To generalize also H\"older source conditions 
\eqref{eq:ssc} with $\nu>1$ to the setting \eqref{eq:genTikh}, Grasmair \cite{Grasmair2013}
imposed a variational source condition on $\pbar$, which turns out to be the solution of 
a Fenchel dual problem. Again the limiting case of this dual source condition, which we tag
\emph{second order source condition}, is equivalent to a simpler condition, 
$T\wbar \in \partial \Sfun^*(\pbar)$, 
which was studied earlier in \cite{Neubauer2010,resmerita:05,RS:06}.  
Hence, Grasmair's second order condition corresponds to the indices $\nu\in (1,2]$ in \eqref{eq:ssc}. 

The aim of this paper is to derive rates of convergence corresponding to indices $\nu>2$, i.e.\ 
faster than $\|\fal-\fdagger\|= \mathcal{O}(\delta^{2/3})$ in a Banach space setting. 
By the well-known saturation effect for Tikhonov regularization \cite{groetsch:84} such rates 
can occur in quadratic Tikhonov regularization only for $\fdagger=0$. 
Therefore, we consider Bregman iterated Tikhonov regularization of the form
\begin{align} \label{Bregit}
	\itfal{n}\in\argmin_{\f\in\X}\left[\frac{1}{\alpha}\Sfun(\T\f-\gobs)
	+\breg{\Rpen}\!\paren{\f,\itfal{n-1}} \right],  \tag{{\new $P_n$}}
\end{align} 
{\new for $n\ge 2$,} which reduces to iterated Tikhonov regularization if $\Rpen(\f) = \|\f\|_{\X}^2$ and 
$\Sfun(\g)=\|\g\|^2$. There is a considerable literature on this type of iteration 
from which we can only give a few references here. Note that for $\Rpen(\f) = \|\f\|_{\X}^2$ 
the iteration \eqref{Bregit} can be interpreted as the proximal point method for minimizing  
$\mathcal{T}(\f):= \Sfun(\T\f-\gobs)$. In \cite{CZ:92,CT:93,eckstein:93} generalizations of the 
proximal point method for general functions $\mathcal{T}$ on $\mathbb{R}^d$ were studied, 
in which the quadratic term is replaced by some Bregman distance (also called $D$-function). 
For $\mathcal{T}(\f)= \Sfun(\T\f-\gobs)$ this leads to \eqref{Bregit}, and the references above 
discuss in particular the case entropy functions $\Rpen$ considered below. In the context of total 
variation regularization of inverse problems, the iteration \eqref{Bregit} was suggested 
 in \cite{Osher2005}. Low order convergence rates of this iterative method 
for quadratic data fidelity terms $\Sfun$ and general penalty terms $\Rpen$ were obtained 
in \cite{BRH:07,FG:12,Frick2011,FS:10}. 
We emphasize that in contrast to all the references above, we consider only small fixed number 
of iterations in \eqref{Bregit} here to cope with the saturation effect. In particular, we 
study convergence in the limit $\alpha\to 0$, rather than $n\to\infty$. 

The main contributions of this paper are:
\begin{itemize}
	\item The formulation of variational source conditions of arbitrarily high order for 
	quadratic regularization in Hilbert spaces (\ref{def:higher_order_VSC}) and the derivation 
	of optimal convergence rates under these conditions (Theorem \ref{Hilbert_theorem}). 
  \item Optimal convergence rates of general Bregman iterated variational regularization \eqref{Bregit} 
	in Banach spaces under a variational source condition of order $3$ (Theorem \ref{rate_theorem}).
  \item Characterization of our new higher order variational source conditions in terms 
	of Besov spaces for finitely smoothing operators, both for quadratic 
	regularization (Corollary \ref{cor:besov_char}) and for maximum entropy regularization 
	(Theorem \ref{thm:ME_rates}). 
\end{itemize}

The remainder of this paper is organized as follows: 
In the following section we review some basic properties of the Bregman iteration \eqref{Bregit} 
and derive a general error bound. The following two sections \S \ref{sec:Hilbert_spaces} 
and \S \ref{sec:third_order} contain our main abstract convergence results in Hilbert and 
Banach spaces, respectively. The following section \S \ref{sec:verification} is devoted 
to the interpretation of higher order variational source conditions. 
Our theoretical results are verified by numerical experiments for entropy regularization in 
\S \ref{sec:numerical}, before we end the paper with some conclusions. 
Some results on duality mappings and consequences of the Xu-Roach inequality are collected 
in an appendix. 

%

\section{Bregman iterations}\label{sec:bregman_it}
Let us first recall the definition of the Bregman distance for a convex functional 
$\Rpen:\X\to (-\infty,\infty]$: Let $\f_0,\f\in\X$ and assume 
that ${\new \f_0}^*\in \X^*$ belongs to the subdifferential of $\Rpen$ at $\f_0$, 
$\f_0^*\in\partial\Rpen(\f_0)$ (see e.g.\ \cite[\S I.5]{ET:76}). Then we set
\[
	\breg{\Rpen}^{f_0^*}(\f,\f_0):=\Rpen(\f)-\Rpen(\f_0)-\lsp f_0^*, \f-\f_0\rsp.
\]
In the context of inverse problems Bregman distances were introduced in \cite{BO:04,eggermont:93}. 
If there is no ambiguity, we sometimes omit the superindex $f_0^*$. This is in particular the case 
if $\Rpen$ is Gateaux differentiable, implying that $\partial\Rpen(\f_0) = \{\Rpen'[\f_0]\}$. 
In the case $\Rpen(\f) =\|\f\|_{\X}^2$ with a Hilbert space $\X$, the Bregman distance 
is simply given by $\breg{\Rpen}(\f,\f_0) = \|\f-\f_0\|_{\X}^2$. In general, however, 
the Bregman distance is neither symmetric nor does it satisfy a triangle inequality. 
Later we will also use symmetric Bregman distances 
$\breg{\Rpen}^{\mathrm{sym},\f_1^*,\f_2^*}(\f_1,\f_2):=\breg{\Rpen}^{\f_1^*}(\f_2,\f_1)+\breg{\Rpen}^{\f_2^*}(\f_1,\f_2)$ for $\f_1,\f_2\in\X$ and $\f_j^*\in \partial \Rpen(\f_j)$, which satisfy  
\begin{align*}
	\symbreg{\Rpen}(\f_1,\f_2)=\lsp \f_2^*-\f_1^*,\f_2-\f_1\rsp.
\end{align*} 
Under the same assumptions the following identity follows from Young's equality:
\begin{align}\label{eq:bregman_identity}
\breg{\Rpen}^{\f_2^*}(\f_1,\f_2) = \breg{\Rpen^*}^{\f_1}(\f_2^*,\f_1^*). 
\end{align}

\medskip
Let us show that Bregman iterations \eqref{Bregit} are well-defined for general data fidelity 
terms of the form \eqref{eq:defiS}. To this end we impose the following conditions: 
\begin{assumption}\label{assump_RS}
Let $\X,\Y$ be Banach spaces, and assume that $\Y$ is $q$-smooth and $\conv$-convex 
$1<q\le 2\le \conv<\infty$ (see Definition \ref{defi:smooth_convex}). 
Moreover, consider an operator $\T\in L(\X,\Y)$, a convex, proper, lower semi-continuous
functional  $\Rpen:\X \to (-\infty,\infty]$ and  $\Sfun$ given by \eqref{eq:defiS}. 
Assume that the functional 
\[
\f\mapsto \frac{1}{\alpha}\Sfun(\T\f-\gobs) + \breg{\Rpen}^{\f_0^*}(\f,\f_0)
\] 
has a unique minimizer for all $(\f_0,\f_0^*)\in \X\times \X^*$ such that 
$\f_0^*\in \partial \Rpen(\f_0)$. 
\end{assumption}

Existence and uniqueness of minimizers has been shown in many cases under different assumptions 
in the literature. As the main focus of this work are convergence rates, we just assume this 
property here. We just mention that it can be shown by a standard argument from calculus of 
variations under the additional assumptions that the sublevel sets 
$\{\f\in \X\colon \breg{\Rpen}^{\f_0^*}(\f,\f_0)\leq M\}$ are weakly {\new or weakly$^*$} sequentially compact 
for all $M\in \mathbb{R}$. 
For $\Rpen$ given by the cross entropy functional discussed in \S \ref{sec:ME} 
such weak continuity of sublevel sets in $L^1$ was shown in \cite[Lemma 2.3]{eggermont:93}. 
For total variation regularization Assumption \ref{assump_RS} has been 
shown in \cite[Prop.~3.1]{Osher2005}. 


For a number $s\in (1,\infty)$ we will denote by $s^*$ the conjugate number satisfying 
$\tfrac{1}{s}+\tfrac{1}{s^*}=1$. Recall that 
\[
\Sfun^*(\p) = \frac{1}{q^*}\norm{\p}_{\Y^*}^{q^*}
\]
and that $\Sfun(\cdot-\gobs)^*(\p) = \Sfun^*(\p)+\langle \p,\gobs\rangle$. {\new The initial step of the Bregman iteration  is} the Tikhonov minimization problem \eqref{eq:genTikh}.
The Fenchel dual to \eqref{eq:genTikh}
is
\begin{align}\label{eq:dualTikh}
\pal\in\argmin_{p\in\Y^*}\left[\frac{1}{\alpha}\Sfun^*(-\alpha p) - \langle p,\gobs\rangle
+\Rpen^*(\T^*p)\right]. \tag{$P_1^*$}
\end{align}
By Theorem 4.1 in \cite[Chap.~III]{ET:76}, $\pal\in\Y^*$ exists as the functional $\Sfun$ 
is continuous everywhere. As $\Y$ is $q$-smooth, $\Y^*$ is $q^*$-convex, and hence 
$\pal$ is unique. 
If $\fal$ exists as well we have strong duality for \eqref{eq:genTikh},\eqref{eq:dualTikh}, 
and by \cite[Chap.~III, Prop.~4.1]{ET:76} the following extremal relations hold true:
\begin{align} \label{IA}
	\T^*\pal\in\partial\Rpen(\fal) \quad\text{ and }\quad -\alpha\pal\in\partial\Sfun(\T\fal-\gobs). 
\end{align}
Using the Bregman distance $\Ritpen{2}(\f):=\breg{\Rpen}^{\T^*\pal}(\f,\fal)$ we can give a 
precise definition of the second step of the Bregman iteration \eqref{Bregit}: 
\begin{align}
		\itfal{2}\in\argmin_{\f\in\X}\left[\frac{1}{\alpha}\Sfun(\T\f-\gobs)+\Ritpen{2}(\f)\right]. \tag{$P_2$}
\end{align}
Like this we can recursively prove well-definedness of the Bregman iteration \eqref{Bregit} as follows:
\begin{prop} \label{prop}
Suppose Assumption \ref{assump_RS} holds true. 
Let $\itfal{1}:=\fal$ be the solution to $(P_1):=\eqref{eq:genTikh}$, and set $\Ritpen{1}:=\Rpen$. 
Then for $n=1,2,\dots$ the dual solutions 
	\begin{align}
\itpal{n}\in\argmin_{p\in\Y^*}\left[\frac{1}{\alpha}\Sfun^*(-\alpha p)
-\langle \p,\gobs\rangle + \Ritpen{n}^*(\T^*p)\right] \tag{$P_n^*$}
	\end{align}
are well defined, and we have strong duality between $(P_n)$ and $(P_n^*)$. Moreover,  
	\begin{align*}
 \f_n^*:=\sum_{k=1}^n \T^*\itpal{k}\in \partial\Rpen\paren{\itfal{n}},
	\end{align*}
such that we can define 
$\Ritpen{n+1}(\f):=\breg{\Rpen}^{f_n^*}(\f,\itfal{n})$ 
	as well as
	\begin{align}
		\itfal{n+1}\in\argmin_{\f\in\X}\left[\frac{1}{\alpha}\Sfun(\T\f-\gobs)+\Ritpen{n+1}(\f)\right]. \tag{$P_{n+1}$}
	\end{align}
\end{prop}
\begin{proof}
	We need to prove the existence of $\itpal{n}$ as well as the fact that  the subdifferential $\partial\Rpen(\itfal{n})$ contains $\sum_{k=1}^n \T^*\itpal{k}$ .
	The existence of $\itpal{n}$ as well as strong duality for $(P_n),(P_n^*)$ follows once again from Theorem 4.1 in \cite{ET:76}.
	The second statement can be proved by induction. The base case was shown in \eqref{IA}.  By strong duality we have
	\begin{align*}
		\T^*\itpal{n}\in\partial\Ritpen{n}(\itfal{n})=\partial\Rpen(\itfal{n})-\sum_{k=1}^{n-1}\T^*\itpal{k}.
	\end{align*}
Therefore, we have $\sum_{k=1}^n \T^*\itpal{k}\in\partial\Rpen(\itfal{n})$ and can define 
$\Ritpen{n+1}$ in the way we claimed.
\qed\end{proof}

A useful fact about the penalty functionals $\Ritpen{n}$ is that their corresponding Bregman 
distances coincide for all $n\in\N$ {\new as they only differ by an affine linear functional: 
\begin{lem} \label{equivalence_of_bregman_dist}
	Let $f_0\in\X,f_0^*\in\partial\Rpen(f_0)$ and $ \tilde{p}:=\sum_{k=1}^{n-1} \itpal{k}$. Then we have
	\begin{align*}
		\breg{\Ritpen{n}}^{\f_0^*-\T^*\tilde{\p}}(\f,f_0)=\breg{\Rpen}^{f_0^*}(\f,\f_0).
	\end{align*}
\end{lem}
\begin{proof}
	By Proposition \ref{prop} we have $\T^*\tilde{\p}\in\partial\Rpen(\itfal{n-1})$ so $\f_0^*-\T^*\tilde{\p}\in\partial\Ritpen{n}(f_0)$ and
\begin{align*} 
	\breg{\Ritpen{n}}^{\f_0^*-\T^*\tilde{\p}}(\f,f_0) \notag
&=\breg{\Rpen}^{\T^*\tilde{\p}}(\f,\itfal{n-1})-\breg{\Rpen}^{\T^*\tilde{\p}}(\f_0,\itfal{n-1})
-\lsp \f_0^*-\T^*\tilde{\p},\f-\f_0\rsp \\
	&=\Rpen(\f)-\Rpen(\f_0)-\lsp \f_0^*,\f-\f_0\rsp=\breg{\Rpen}^{f_0^*}(\f,\f_0).
\end{align*}
\qed\end{proof}}

The following lemma describes the first step towards our bounds on the error in the 
Bregman distance. All that is then left is to construct appropriate vectors $\f$ which 
approximately minimize the functional on the right hand side. 
\begin{lem} \label{Lemma11.19}
Suppose Assumption \ref{assump_RS} holds true and there exists 	
$\pbar\in\Y^*$ such that $\T^*\pbar\in\partial\Rpen(\fdagger)$. 
With the notation of Proposition \ref{prop} 
define $s^{(n)}_{\alpha}:= \pbar-\sum_{k=1}^{n-1}\itpal{k}$. Then
	\begin{align*}
\breg{\Rpen}^{\T^*\pbar}\paren{\itfal{n},\fdagger}
\le \inf_{f\in \X}\bigg[\frac{1}{\alpha}\Sfun\paren{\T\f-\gobs}+\lsp s^{(n)}_{\alpha},\T\f-\gobs  \rsp& \\
+\frac{1}{\alpha}\Sfun^*\left(-\alpha s^{(n)}_{\alpha}\right) 
+\breg{\Rpen}^{\T^*\pbar}\paren{\f,\fdagger} & \bigg].
	\end{align*}
\end{lem}

\begin{proof}
Due to the minimizing property of $\itfal{n}$ we have 
\begin{align*}
\frac{1}{\alpha}\Sfun\paren{\T\itfal{n}-\gobs}+\Ritpen{n}\paren{\itfal{n}}
\le\frac{1}{\alpha}\Sfun\paren{\T\f-\gobs}+\Ritpen{n}(\f),
\end{align*}
for all $\f\in\X$, which is equivalent to
	\begin{align} \label{eq:minimizing_property}
\Ritpen{n}\paren{\itfal{n}}-\Ritpen{n}(\f)\le\frac{1}{\alpha}\Sfun\paren{\T\f-\gobs}
-\frac{1}{\alpha}\Sfun\paren{\T\itfal{n}-\gobs}.
	\end{align}
As $\T^*s^{(n)}_{\alpha}= \T^*\pbar-\f^*_{n-1} 
\in \partial\Rpen(\fdagger)-\f^*_{n-1}=\partial\Ritpen{n}(\fdagger)$  
by Proposition \ref{prop}, it follows that 
	\begin{align*}
\breg{\Ritpen{n}}^{\T^*s^{(n)}_{\alpha}}\paren{\itfal{n},\fdagger}
&=\Ritpen{n}\paren{\itfal{n}}-\Ritpen{n}\paren{\fdagger}-\lsp \T^*s^{(n)}_{\alpha},\itfal{n}-\fdagger\rsp \\
&\le \frac{1}{\alpha}\Sfun(\T\f-\gobs)-\frac{1}{\alpha}\Sfun\paren{\T\itfal{n}-\gobs}  \\
&\quad -\lsp \T^*s^{(n)}_{\alpha},\itfal{n}-\fdagger\rsp+\Ritpen{n}(\f)-\Ritpen{n}\paren{\fdagger}.
	\end{align*}
	Due to the strong duality (Propostion \ref{prop}) the extremal relation 
$-\alpha\itpal{n}\in\partial\Sfun(\T\itfal{n}-\gobs)$ holds true, and thus the generalized Young 
equality yields
\begin{align*}
-\frac{1}{\alpha}\Sfun\paren{\T\itfal{n}-\gobs}
&=\frac{1}{\alpha}\Sfun^*\paren{-\alpha\itpal{n}}+\lsp \itpal{n},\T\itfal{n}-\gobs\rsp \\
&=\frac{1}{\alpha}\Sfun^*\left(-\alpha s^{(n)}_{\alpha}\right)+\lsp s^{(n)}_{\alpha},\T\itfal{n}-\gobs\rsp \\
& \quad -\frac{1}{\alpha}\breg{\Sfun^*}\left(-\alpha s^{(n)}_{\alpha},-\alpha\itpal{n}\right)\\
&\le\frac{1}{\alpha}\Sfun^*\left(-\alpha s^{(n)}_{\alpha}\right)+\lsp s^{(n)}_{\alpha},\T\itfal{n}-\gobs\rsp
\end{align*}
where we have used that  the Bregman distance is non-negative. 
Combining this gives
\begin{align*}
\breg{\Ritpen{n}}^{\T^*s^{(n)}_{\alpha}}\paren{\itfal{n},\fdagger}
&\le\frac{1}{\alpha}\Sfun(\T\f-\gobs)+\frac{1}{\alpha}\Sfun^*\left(-\alpha s^{(n)}_{\alpha}\right) 
+\lsp s^{(n)}_{\alpha},\T\fdagger-\gobs\rsp \\
& \quad +\Ritpen{n}(\f)-\Ritpen{n}\paren{\fdagger} \\
&= \frac{1}{\alpha}\Sfun(\T\f-\gobs)+\lsp s^{(n)}_{\alpha},\T\f-\gobs\rsp  \\
& \quad +\frac{1}{\alpha}\Sfun^*\left(-\alpha s^{(n)}_{\alpha}\right) 
+\breg{\Ritpen{n}}^{\T^*s^{(n)}_{\alpha}}\paren{\f,\fdagger}.
\end{align*}
Now the identity 
$\breg{\Ritpen{n}}^{\T^*s^{(n)}_{\alpha}}(\f,\fdagger)=\breg{\Rpen}^{\T^*\pbar}(\f,\fdagger)$  shown in Lemma \ref{equivalence_of_bregman_dist} completes the proof.
\qed\end{proof}

\section{Higher order variational source conditions in Hilbert spaces} \label{sec:Hilbert_spaces}

We will now go back to the Hilbert space setting, where $\X,\Y$ are Hilbert spaces and $\Rpen(\f):=\tfrac{1}{2}\|\f\|_{\X}^2$, $\Sfun(g)=\frac{1}{2}\norm{g}_\Y^2$, to prove \eqref{classical_opt_order_rates} using variational source conditions, which are defined as follows: 
\begin{defi}[Variational source condition $\VSC^l(\fdagger,\Phi)$] \label{def:higher_order_VSC}
Let $\fdagger\in\HilbertX$, let $\Phi$ be a concave index function, and let $n\in\N$. 
Then the statement  
	\begin{align}\label{VSC^2n-1} 
	\begin{split}
&\exists\itwbar{n-1}\in\HilbertX\,: \fdagger=(\T^*\T)^{n-1}\itwbar{n-1}		 \\
\land \quad &\forall  \f\in\HilbertX\,:\,
\lsp  \itwbar{n-1},\f \rsp\le\frac{1}{2}\norm{\f}^2+\Phi\left(\norm{\T\f}^2\right)
\end{split}
	\end{align}
will be abbreviated by $\VSC^{2n-1}(\fdagger,\Phi)$, and the statement 
	\begin{align}\label{VSC^2n}
	\begin{split}
	&\exists \itpbar{n}\in\HilbertY\, : \, \fdagger=(\T^*\T)^{n-1}\T^*\itpbar{n} \\
	\land\quad &\forall \p\in\HilbertY\,:\,
		\lsp \p,\itpbar{n} \rsp\le\frac{1}{2}\norm{\p}^2+\Phi\left(\norm{\T^*\p}^2\right)
		\end{split}
	\end{align}
will be abbreviated by  $\VSC^{2n}(\fdagger,\Phi)$. $\VSC^l(\fdagger,\Phi)$ for 
$l\in\N$ will be referred to as \emph{variational source condition of order $l$ 
with index function $\Phi$} for (the true solution) $\fdagger$. 
\end{defi}

Note that that $\VSC^1(\fdagger,\Phi)$ is  the classical variational source condition,  and $\VSC^2(\fdagger,\Phi)$ coincides with the source condition from 
\cite{Grasmair2013} up to the term $\frac{1}{2}\norm{\p}^2$, which implies that 
$\VSC^2(\fdagger,\Phi)$ is formally weaker than the condition in \cite{Grasmair2013}. 
It is well known that the spectral H\"older source condition \eqref{eq:ssc} with 
$\nu\in (0,1]$ implies 
$\VSC^1(\fdagger,A\id^{\nu/(\nu+1)})$ for some $A>0$ and $\id(t):=t$ (see \cite{HY:10}).
Therefore, it is easy to see that for any $l\in\N$ and $\nu\in[0,1]$ the implication  
\begin{equation}\label{eq:ssc_imply_vsc}
\fdagger\in \mathrm{ran}\paren{(\T^*\T)^{\frac{l-1+\nu}{2}}}
\quad \Rightarrow\quad 
\exists A>0\,:\, \VSC^l\paren{\fdagger, A\id^{\frac{\nu}{\nu+1}}}
\end{equation}
holds true. The converse implication is false for $\nu\in (0,1)$ as discussed 
in \S \ref{sec:verificationHilbert}. 
For $\nu=1$ we have by \cite[Prop.~3.35]{Scherzer2008} that 
\begin{equation}\label{eq:ssc_equiv_vsc}
\fdagger\in \mathrm{ran}\paren{(\T^*\T)^{\frac{l}{2}}}
\quad \Leftrightarrow\quad 
\exists A>0\,:\, \VSC^l\paren{\fdagger, A\sqrt{\cdot}}.
\end{equation}
The aim of this section is to prove error bounds for iterated Tikhonov regularization 
based on these source conditions: 

\begin{thm} \label{Hilbert_theorem}
Let $l,m\in\N$ with $m\geq l/2$, let $\Phi$ be an index function, and 
let $\psi(s):=\sup_{t\geq 0}[st+\Phi(s)]$ denote the Fenchel conjugate of 
$-\Phi$. If $\VSC^l(\fdagger,\Phi)$ holds true, there exists a constant $C>0$ 
depending only on $l$ and $m$ such that  
	\begin{align}\label{l_estimate}
		\norm{\itfal{m}-\fdagger}^2\leq C\paren{\frac{\delta^2}{\alpha}+\alpha^{l-1}\psi\left(-\frac{1}{\alpha}\right)}\qquad \mbox{for all }\alpha,\delta> 0.
	\end{align}
\end{thm}

\begin{proof} We choose $n\in\N$ such that $l=2n$ or $l=2n-1$. Then $m\geq n$. 
In the following $C$ will denote a generic constant depending only on $m$ and $l$. 
The proof proceeds in four steps.\\
\emph{Step 1: Reduction to the case $m=n$.} 
	By Proposition \ref{prop} and the definition of the Bregman distance we have for all $k\ge 2$ and $\f\in\X$ that
	\begin{align*}
		\breg{\Ritpen{k}}(\itfal{k},\f)=\Ritpen{k}(\itfal{k})-\Ritpen{k}(\f)
		-\lsp \f-\sum_{j=1}^{k-1}\T^*\itpal{j}, \itfal{k}-\f \rsp.
	\end{align*}	
	By the optimality condition $\sum_{j=1}^{k-1}\T^*\itpal{j}=\itfal{k-1}$ and the minimizing property of $\itfal{k}$ \eqref{eq:minimizing_property} we have 
	\begin{align} 
	\begin{split}\label{eq:recursive_bound}
		\frac{1}{2}\norm{\itfal{k}-\f}^2=\breg{\Ritpen{k}}(\itfal{k},\f)\le&\frac{1}{2\alpha}\left(\norm{\T\f-\gobs}^2-\norm{\T\itfal{k}-\gobs}^2\right) \\
		&-\lsp \f-\itfal{k-1}, \itfal{k}-\f \rsp. 
	\end{split}
	\end{align}
	Choosing $\f=\fdagger$ gives 
	\begin{align*}
		\frac{1}{2}\norm{\itfal{k}-\fdagger}^2\le \frac{1}{2\alpha}\norm{\gdagger-\gobs}^2 +\norm{\itfal{k-1}-\fdagger}\norm{\itfal{k}-\fdagger}.
	\end{align*}
	Multiplying by four, subtracting $\norm{\itfal{k}-\fdagger}^2$ on both sides and completing the square we get
	\begin{align*}
		\norm{\itfal{k}-\fdagger}^2\le \frac{2\delta^2}{\alpha}+4\norm{\itfal{k-1}-\fdagger}^2.
	\end{align*}
	So it is enough to prove \eqref{l_estimate} for $m=n$ as this will then also imply the claimed error bound for all $m\ge n$ by the above inequality.
	
\medskip	
\emph{Step 2: Error decomposition based on Lemma \ref{Lemma11.19}.} 
Both Assumptions \eqref{VSC^2n-1} and \eqref{VSC^2n} imply that there exist 
$\itpbar{1},\dots,\itpbar{n-1}\in\HilbertY,\itwbar{1},\dots,\itwbar{n-1}\in \HilbertX$ such that 
$\fdagger=(\T^*\T)^{j-1}\T^*\itpbar{j},\fdagger=(\T^*\T)^{j}\itwbar{j}$ for $j=1,\dots,n-1$. 
In the following we will write  $\itpbar{1}=\pbar$ and $\itwbar{1}=\wbar$. 
We have $\partial\Rpen(\fdagger)=\{\fdagger\}=\{\T^*\pbar\}$, so Lemma \ref{Lemma11.19} yields 
	\begin{align}
		\norm{\itfal{n}-\fdagger}^2 &\le \frac{1}{\alpha}\left(\norm{\T\f-\gobs}^2+2\alpha\lsp s^{(n)}_{\alpha},\T\f-\gobs\rsp 
		 +\norm{-\alpha s^{(n)}_{\alpha}}^2\right) \notag \\
		&\quad+\norm{\f-\fdagger}^2  \label{eq:polarization_id}\\
		&= \frac{1}{\alpha}\norm{\T\f-\gobs+\alpha\left(\pbar-\sum_{k=1}^{n-1}\itpal{k}\right)}^2+\norm{\f-\fdagger}^2 \notag
	\end{align}
for $s^{(n)}_{\alpha}=\pbar-\sum_{k=1}^{n-1}\itpal{k}$ and all $\f\in\X$. 

We will choose $\f=n\fdagger-\alpha\wbar-\sum_{k=1}^{n-1}\itfal{k}$. 
As $\T\wbar=\pbar$ and $\T\itfal{k}-\gobs=-\alpha\itpal{k}$ by strong duality, we have
	\begin{align} \label{bound_in_terms_of_itfalk}
		\norm{\itfal{n}-\fdagger}^2 &\le  \frac{1}{\alpha}\norm{n\paren{\gdagger-\gobs}}^2+\norm{(n-1)\fdagger-\alpha\wbar-\sum_{k=1}^{n-1}\itfal{k}}^2.
	\end{align}
	It remains to bound the second term, which does not look favourable at first sight as we 
know that $\norm{\itfal{k}-\fdagger}$ should converge to zero slower than 
$\norm{\itfal{n}-\fdagger}$ for $k<n$. But it turns out that we have cancellation between the different $\itfal{k}$. Therefore, we will now 
introduce vectors $\sigma_k\in\HilbertX$ such that 
		\begin{align}\label{cancellation_bound}
		\norm{(n-1)\fdagger-\alpha\wbar-\sum_{k=1}^{n-1}\itfal{k}}
		\le \sum_{k=1}^{n-1}\norm{\itfal{k}-\fdagger-\sigma_k}
	\end{align}
	and then prove that all terms on the right hand side are of optimal order. 

Let $(b_{kj})\in\N^{\N\times\N}$ denote the matrix given by Pascal's triangle, i.e. $b_{k,j}=\binom{k+j-2}{j-1}$ for all $k,j\in\N$. We will 
need the identities
	\begin{align} \label{binomial_theorem}
		\sum_{k+j=n}(-1)^jb_{k,j}=-\delta_{n-2,0}\qquad \mbox{for all }n\geq 2,
	\end{align}
	which are equivalent to $(1-1)^{n-2}=\delta_{n-2,0}$ by the binomial theorem 
$(a+b)^n=\sum_{k+j=n+2}b_{k,j}a^{k-1}b^{j-1}$. 
Moreover, we need the defining property of the triangle,
	\begin{align} \label{pascal_property}
		b_{k,j}+b_{k-1,j+1}=b_{k,j+1}.
	\end{align}
	Using \eqref{binomial_theorem} we can add zero in the form 
	\begin{align*}
0=\alpha\wbar+\sum_{l=1}^{n-1}\alpha^l\itwbar{l}\sum_{\mathclap{ k+j=l+1}}(-1)^jb_{k,j}=\alpha\wbar + \sum_{k=1}^{n-1}\sum_{j=1}^{n-k}(-1)^jb_{k,j}\alpha^{k+j-1}\itwbar{k+j-1}
	\end{align*}
	to find that
	\begin{align*}
		(n-1)\fdagger-\alpha\wbar-\sum_{k=1}^{n-1}\itfal{k}
		=\sum_{k=1}^{n-1}\left(\fdagger-\itfal{k}+\sum_{j=1}^{n-k}(-1)^jb_{k,j}\alpha^{k+j-1}\itwbar{k+j-1}\right)
	\end{align*}
and by the triangle inequality this yields \eqref{cancellation_bound} with
\[
	\sigma_k:=\sum_{j=1}^{n-k}(-1)^jb_{k,j}\alpha^{k+j-1}\itwbar{k+j-1},\qquad 
	k\in\N.
\]
It will be convenient to set $\sigma_0:= -\fdagger$ and $\itfal{0}:=0$.  

\medskip
\emph{Step 3: proof of \eqref{l_estimate} for the case $l=2n-1$.} In view of 
\eqref{bound_in_terms_of_itfalk} and \eqref{cancellation_bound} it suffices 
to prove by induction that given $\VSC^{2n-1}(\fdagger,\Phi)$ \eqref{VSC^2n-1} we have 
	\begin{align} \label{induction_hypothesis}
		\norm{\itfal{k}-\fdagger-\sigma_k}^2\leq C\paren{\frac{\delta^2}{\alpha}+\alpha^{2n-2}\psi\left(\frac{-1}{\alpha}\right)},\quad k=0,1,\dots, n-1. 
	\end{align}
For $k=0$ this is trivial. Assume now that \eqref{induction_hypothesis} holds 
true for $k-1$ with $k\in \{1,\dots,n-1\}$. 
Insert $\f=\fdagger+\sigma_k$ in \eqref{eq:recursive_bound} to get
	\begin{align*}
		\norm{\itfal{k}-\fdagger-\sigma_k}^2\le&\frac{1}{\alpha}\left(\norm{\gdagger+\T\sigma_k-\gobs}^2-\norm{\T\itfal{k}-\gobs}^2\right) \\
		&-2\lsp \fdagger+\sigma_k-\itfal{k-1}, \itfal{k}-\fdagger-\sigma_k \rsp. 
	\end{align*}
	Now we add and subtract $\itfal{k-1}-\fdagger-\sigma_{k-1}$ to the first 
	term of the inner product to find
\begin{align*}
		\norm{\itfal{k}-\fdagger-\sigma_k}^2 &\le \frac{1}{\alpha}\left(\norm{\gdagger+\T\sigma_k-\gobs}^2-\norm{\T\itfal{k}-\gobs}^2\right) \\
		& \quad -2\lsp \sigma_k-\sigma_{k-1}, \itfal{k}-\fdagger-\sigma_k \rsp \\
		& \quad +2\norm{\itfal{k-1}-\fdagger-\sigma_{k-1}}\norm{\itfal{k}-\fdagger-\sigma_k} 
	\end{align*}
The last term, denoted by {\new \[E:=2\norm{\itfal{k-1}-\fdagger-\sigma_{k-1}}\norm{\itfal{k}-\fdagger-\sigma_k},\] } will be dealt with at the end of this step.	Because of the identity 
$T^*T\omega^{(l)} =  \omega^{(l-1)}$ and \eqref{pascal_property} we have
	\begin{align*}
		\sigma_k-\sigma_{k-1}&=\alpha^{k-1}\itwbar{k-1}+\sum_{j=1}^{n-k}(-1)^j(b_{k,j}+b_{k-1,j+1})\alpha^{k+j-1}\itwbar{k+j-1} \\
		&=-\frac{1}{\alpha}\T^*\T\sigma_k+(-1)^{n-k}b_{k,n-k+1}\alpha^{n-1}\itwbar{n-1}
	\end{align*}
for $k>1$, and it is easy to see that this also holds true for $k=1$. 
	Therefore, 
	\begin{multline*}
		\lsp \sigma_k-\sigma_{k-1}, \itfal{k}-\fdagger-\sigma_k \rsp= \frac{1}{\alpha}\lsp \T\sigma_k,\gdagger+ \T\sigma_k-\gobs+\gobs-\T\itfal{k} \rsp \\
		+\lsp (-1)^{n-k}b_{k,n-k+1}\alpha^{n-1}\itwbar{n-1}, \itfal{k}-\fdagger-\sigma_k \rsp,
	\end{multline*}
	which yields
	\begin{align} \label{classical_VSC_reasoning}
		\norm{\itfal{k}-\fdagger-\sigma_k}^2 &\le\frac{1}{\alpha}\left(\norm{\gdagger-\gobs}^2-\norm{\T\itfal{k}-\T\sigma_k-\gobs}^2\right) +E  \\
		& \quad-(-1)^{n-k}2 b_{k,n-k+1}\alpha^{n-1}\lsp \itwbar{n-1}, \itfal{k}-\fdagger-\sigma_k \rsp	\notag
	\end{align}
	For shortage of notation denote $b=2 b_{k,n-k+1}$. Apply $\VSC^{2n-1}(\fdagger,\Phi)$ \eqref{VSC^2n-1} with $\f=(-1)^{n-k}(\itfal{k}-\fdagger-\sigma_k)/(b\alpha^{n-1})$ and multiply by $(b\alpha^{n-1})^2$ to obtain 
	\begin{multline*}
		-(-1)^{n-k} 2 b_{k,n-k+1}\alpha^{n-1}  \lsp \itwbar{n-1}, \itfal{k}-\fdagger-\sigma_k \rsp \\
		\le \frac{1}{2}\norm{\itfal{k}-\fdagger-\sigma_k}^2+(b\alpha^{n-1})^2\Phi\left((b\alpha^{n-1})^{-2}\norm{\T\itfal{k}-\gdagger-\T\sigma_k}^2\right).
	\end{multline*}
	Combining this bound with \eqref{classical_VSC_reasoning} yields 
	\begin{align*}
		\frac{1}{2}\norm{\itfal{k}-\fdagger-\sigma_k}^2 \le&\frac{1}{\alpha}\left(\norm{\gdagger-\gobs}^2-\norm{\T\itfal{k}-\T\sigma_k-\gobs}^2\right) +E \\
		&+(b\alpha^{n-1})^2\Phi\left((b\alpha^{n-1})^{-2}\norm{\T\itfal{k}-\gdagger-\T\sigma_k}^2\right).
	\end{align*}
	 Then we have
{\allowdisplaybreaks	\begin{align*}
		\frac{1}{2}\norm{\itfal{k}-\fdagger-\sigma_k}^2 \le &\frac{1}{\alpha}\left(\norm{\gdagger-\gobs}^2-\norm{\T\itfal{k}-\T\sigma_k-\gobs}^2\right) +E \\
		& \quad +(b\alpha^{n-1})^2\Phi\left((b\alpha^{n-1})^{-2}\norm{\T\itfal{k}-\gdagger-\T\sigma_k}^2\right) \\
		& {\new\le}\frac{\delta^2}{\alpha}-\frac{1}{\alpha}\norm{\T\itfal{k}-\T\sigma_k-\gdagger}^2 +E \\
		& \quad +(b\alpha^{n-1})^2\Phi\left((b\alpha^{n-1})^{-2}\norm{\T\itfal{k}-\gdagger-\T\sigma_k}^2\right) \\
		&\le \frac{\delta^2}{\alpha}+b^2\alpha^{2n-2}\sup_{\tau\ge 0}\left[\frac{-\tau}{\alpha}-\left(-\Phi(\tau)\right)\right] +E \\
		&=\frac{\delta^2}{\alpha}+4 b_{k,n-k+1}^2\alpha^{2n-2}\psi\left(\frac{-1}{\alpha}\right)+E.
	\end{align*}}%
	To get rid of $E=2\norm{\itfal{k-1}-\fdagger-\sigma_{k-1}}\norm{\itfal{k}-\fdagger-\sigma_k}$ subtract the term $\frac{1}{4}\norm{\itfal{k}-\fdagger-\sigma_k}^2$ on both sides and use Young's inequality as well as the induction hypothesis \eqref{induction_hypothesis}.
	
\medskip
	\emph{Step 4: Proof of \eqref{l_estimate} for the case $l=2n$.} In view of 
	\eqref{bound_in_terms_of_itfalk} and \eqref{cancellation_bound} 
it suffices to prove by induction that given $\VSC^{2n}(\fdagger,\Phi)$ \eqref{VSC^2n} we have  
	\begin{align} \label{induction_hypothesis2}
		\norm{\itfal{k}-\fdagger-\sigma_k}^2\leq C\paren{\frac{\delta^2}{\alpha}+\alpha^{2n-1}\psi\left(-\frac{1}{\alpha}\right)}, \qquad 
		k=0,\dots,n-1.
	\end{align}
	Again, the case $k=0$ is trival. Assume that \eqref{induction_hypothesis2} holds true for all $j=1,\dots, k-1$. Note that 
	\begin{align*}
		&\norm{\itfal{k}-\fdagger-\sigma_k}^2 =\lsp\itfal{k}-\fdagger-\sigma_k, \itfal{k}-\fdagger-\sigma_k\rsp \\
		& \le \lsp\itfal{k}-\fdagger-\sigma_k, \sum_{j=1}^k \left(\itfal{j}-\fdagger-\sigma_j \right)\rsp \\
		& \quad +\sum_{j=1}^{k-1}\norm{\itfal{k}-\fdagger-\sigma_k}\norm{\itfal{j}-\fdagger-\sigma_j }.
	\end{align*}
	Then Young's inequality together with the induction hypothesis 
	\eqref{induction_hypothesis2}  gives 
	\begin{align} \label{bound_after_applying_induction_hypo}
	\begin{split}
		\frac{1}{2}\norm{\itfal{k}-\fdagger-\sigma_k}^2  &\le \lsp\itfal{k}-\fdagger-\sigma_k, \sum_{j=1}^k \left(\itfal{j}-\fdagger-\sigma_j \right)\rsp \\
		& \quad +C\paren{\frac{\delta^2}{\alpha}+\alpha^{2n-1}\psi\left(-\frac{1}{\alpha}\right)}.
		\end{split}
	\end{align}
A simple computation (for example another induction) shows that 
	\begin{align*}
		-\alpha\wbar-\sum_{j=1}^k \sigma_j= \sum_{j=1}^{n-k-1}(-1)^jb_{k,j}\alpha^{k+j}\itwbar{k+j}=:\hat{\sigma}_k.
	\end{align*}
	{\new By $\VSC^{2n}(\fdagger,\Phi)$ \eqref{VSC^2n} we have $\sigma_k\in\range T^*$} and by Proposition \ref{prop} we have $\T^*\itpal{k}\in\partial\Ritpen{k}(\itfal{k})=\{\itfal{k}-\sum_{j=1}^{k-1}\T^*\itpal{j}\}$ as well as
	$-\alpha\itpal{j}\in\partial\Sfun(\T\itfal{j}-\gobs)=\{\T\itfal{j}-\gobs\}$ such that 
	\begin{align*}
		&\lsp\itfal{k}-\fdagger-\sigma_k, \sum_{j=1}^k \left(\itfal{j}-\fdagger-\sigma_j \right)\rsp \\
		=&\lsp\sum_{j=1}^{k}\T^*\itpal{j}-\T^*\pbar-\T^*(\T^{*-1}\sigma_k), \sum_{j=1}^k \itfal{j} -k\fdagger+\alpha\wbar+\hat{\sigma}_k\rsp  \\
		=&\lsp\sum_{j=1}^{k}\itpal{j}-\pbar-(\T^{*-1}\sigma_k), \sum_{j=1}^k (-\alpha\itpal{j}) +\alpha\T\wbar+\T\hat{\sigma}_k+k(\gobs-\gdagger)\rsp \\
		= &\alpha\lsp \pbar+(\T^{*-1}\sigma_k)-\sum_{j=1}^{k}\itpal{j}, \sum_{j=1}^k \itpal{j} -\pbar-\frac{\T\hat{\sigma}_k}{\alpha}\rsp+ kE,
	\end{align*}
	where $E:=\lsp \pbar+(\T^{*-1}\sigma_k)-\sum_{j=1}^{k}\itpal{j}, \gdagger-\gobs \rsp$. On the right hand side of the scalar product we now exchange  $\sum_{j=1}^k \itpal{j} -\pbar$ by $(\T^{*-1}\sigma_k)$  to find
	\begin{align*}
		&\lsp \pbar+(\T^{*-1}\sigma_k)-\sum_{j=1}^{k}\itpal{j}, \sum_{j=1}^k \itpal{j} -\pbar-\frac{\T\hat{\sigma}_k}{\alpha}\rsp  \\
		=& \lsp \pbar+(\T^{*-1}\sigma_k)-\sum_{j=1}^{k}\itpal{j}, (\T^{*-1}\sigma_k)-\frac{\T\hat{\sigma}_k}{\alpha}\rsp 
		-\norm{\pbar+(\T^{*-1}\sigma_k)-\sum_{j=1}^{k}\itpal{j}}^2 
	\end{align*}
and together with the identity 
	\begin{align*}
		(\T^{*-1}\sigma_k)-\frac{\T\hat{\sigma}_k}{\alpha}&=\sum_{j=1}^{n-k}(-1)^jb_{k,j}\alpha^{k+j-1}\itpbar{k+j}+\sum_{j=1}^{n-k-1}(-1)^jb_{k,j}\alpha^{k+j-1}\itpbar{k+j} \\
		&=(-1)^{n-k}b_{k,n-k}\alpha^{n-1}\itpbar{n}
	\end{align*}
	it follows that 
	\begin{multline}	 \label{allinall}
		\lsp\itfal{k}-\fdagger-\sigma_k, \sum_{j=1}^k \left(\itfal{j}-\fdagger-\sigma_j \right)\rsp = 
		-\alpha\norm{\pbar+(\T^{*-1}\sigma_k)-\sum_{j=1}^{k}\itpal{j}}^2 \\
		+b_{k,n-k}\alpha^n\lsp (-1)^{n-k}\left(\pbar+(\T^{*-1}\sigma_k)-\sum_{j=1}^{k}\itpal{j}\right),\itpbar{n} \rsp+kE , 
	\end{multline}
	so we are finally in a position to apply $\VSC^{2n-1}(\fdagger,\Phi)$ \eqref{VSC^2n}. 
	For shortage of notation denote {\new $\tilde{b}=4b_{k,n-k}$} and $\tilde{p}=\pbar+(\T^{*-1}\sigma_k)-\sum_{j=1}^{k}\itpal{j}$. Choose $\p=(-1)^{n-k}\tilde{p}/(\tilde{b}\alpha^{n-1})$, and multiply 
	the inequality by $\alpha(\tilde{b}\alpha^{n-1})^2$  to obtain 
	\begin{align} \label{applied_VSC^2n}
	\begin{aligned}
		&4\tilde{b}_{k,n-k}\alpha^n\lsp (-1)^{n-k}\left(\pbar+(\T^{*-1}\sigma_k)-\sum_{j=1}^{k}\itpal{j}\right),\itpbar{n} \rsp \\
		&\le \frac{\alpha}{2}\norm{\tilde{p}}^2 +\tilde{b}^2\alpha^{2n-1}\Phi\left((\tilde{b}\alpha^{n-1})^{-2}\norm{\T^*\tilde{p}}^2\right).
		\end{aligned}
	\end{align}	
	Now combine \eqref{bound_after_applying_induction_hypo}, \eqref{allinall} and \eqref{applied_VSC^2n} to find
	\begin{multline} 
		2\norm{\itfal{k}-\fdagger-\sigma_k}^2 
		\leq \tilde{b}^2\alpha^{2n-1} \Phi\left((\tilde{b}\alpha^{n-1})^{-2}\norm{\fdagger+\sigma_k-\itfal{k}}^2\right) \\ 
		+\frac{\alpha}{2}\norm{\tilde{p}}^2-4\alpha\norm{\tilde{p}}^2 + 4kE 
		+C\paren{\frac{\delta^2}{\alpha}+\alpha^{2n-1}\psi\left(-\frac{1}{\alpha}\right)}. 
		\label{close_to_finished}
	\end{multline}
	Completing the square, we get  
	\begin{align*}
		\frac{\alpha}{2}\norm{\tilde{p}}^2-4\alpha\norm{\tilde{p}}^2 + 4kE 
		= -\frac{7}{2}\alpha\norm{\tilde{p}}^2
		+ 4\lsp \tilde{p},\gdagger-\gobs\rsp 
		\leq \frac{8k^2\delta^2}{7\alpha}.
	\end{align*}
Now we subtract 
$\big\|\itfal{k}-\fdagger-\sigma_k\big\|^2$ in \eqref{close_to_finished} from both sides to find
	\begin{align*}
		\norm{\itfal{k}-\fdagger-\sigma_k}^2 &\le \tilde{b}^2\alpha^{2n-1} \Phi\left((\tilde{b}\alpha^{n-1})^{-2}\norm{\fdagger+\sigma_k-\itfal{k}}^2\right)  \\
		&\quad -\norm{\itfal{k}-\fdagger-\sigma_k}^2+C\paren{\frac{\delta^2}{\alpha}+\alpha^{2n-1}\psi\left(-\frac{1}{\alpha}\right)} \\
		\le& \tilde{b}^2   \alpha^{2n-1} \sup_{\tau\ge 0}\left[\frac{-\tau}{\alpha}-(-\Phi\left(\tau\right))\right]
		+C\paren{\frac{\delta^2}{\alpha}+\alpha^{2n-1}\psi\left(-\frac{1}{\alpha}\right)} \\
		=& 16b_{k,n-k}^2\alpha^{2n-1} \psi\left(\frac{-1}{\alpha}\right)
		+C\paren{\frac{\delta^2}{\alpha}+\alpha^{2n-1}\psi\left(-\frac{1}{\alpha}\right)}.\qquad\qed
	\end{align*}
\end{proof}

Note that under a spectral source condition as on the left hand side of the 
implication \eqref{eq:ssc_imply_vsc}, the VSC of the right hand side of 
\eqref{eq:ssc_imply_vsc} and Theorem \ref{Hilbert_theorem} yield the  
error bound $C(\delta/\alpha^2+\alpha^{l-1+\nu})$. For the choice 
$\alpha\sim \delta^{2/(l+\nu)}$ this leads to the optimal convergence 
rate $\big\|\itfal{m}-\fdagger\big\|= \landauO{\delta^{(l-1+\nu)/(l+\nu)}}$. 
However, we have derived this rate under the weaker assumption 
$\VSC^l(\fdagger,A\id^{\nu/(\nu+1)})$ using only variational, but no spectral arguments. 

\section{Higher order convergence rates in Banach spaces} \label{sec:third_order}

In this section we will introduce a third order version of the variational source condition 
\eqref{eq:vsc1} in Banach spaces. Let us abbreviate \eqref{eq:vsc1} 
by $\VSC^1(\fdagger,\Phi,\Rpen,\Sfun)$ in the following. 
First we give a definition for the second order source condition in Banach spaces 
based on \cite[(4.2)]{Grasmair2013}.

\begin{defi}[Variational source condition $\VSC^2(\fdagger,\Phi,\Rpen,\Sfun)$] \label{VSC^2}  Let $\Phi$ 
 be an index function and $\Rpen$ a proper, convex, lower-semicontinuous functional on $\X$.  We say that $\fdagger\in \X$ satisfies 
the second order variational source condition $\VSC^2(\fdagger,\Phi,\Rpen,\Sfun)$ if 
there exist $\pbar\in\Y^*$ such that $\T^*\pbar\in\partial\Rpen(\fdagger)$ and $\pstar\in \partial \Sfun^*(\pbar)$ such that
	\begin{align}\label{eq:VSC^2_inequality}
	\forall p\in\Y^*: \qquad
	\lsp \pbar-p,\pstar\rsp \le \frac{1}{2}\breg{\Sfun{}^*}^{\pstar}\!\paren{p,\pbar} 
	+ \Phi\left(\breg{\Rpen^*}^{\fdagger}\!\paren{T^*p,T^*\pbar} \right).
	\end{align}
\end{defi}

\begin{rem}
	Let Assumption \ref{assump_RS} hold. Then $\partial\Sfun^*(\pbar)=\{J_{q^*,\Y^*}(\pbar)\}$, with $J_{q^*,\Y^*}$ being the duality mapping defined in the appendix. So $\VSC^2(\fdagger,\Phi,\Rpen,\Sfun)$ is equivalent to \cite[(4.2)]{Grasmair2013} up to the additional term $\frac{1}{2}\breg{\Sfun^*}\!\paren{p,\pbar}$. It is easy to see from the proof of \cite[Theorem 4.4]{Grasmair2013}, that one still can conclude convergence rates
	\begin{align} \label{eq:Grasmairs_rates}
		\breg{\Rpen}\!\paren{\fal,\fdagger}\le \alpha^{q^*-1}(-\Phi)^*\paren{-1/\alpha^{q^*-1}}+\tilde{D}\frac{\delta^q}{\alpha},
	\end{align}
	with a slightly changed constant $\tilde{D}>0$.
\end{rem}	

\begin{defi}[Variational source condition $\VSC^3(\fdagger,\Phi,\Rpen,\Sfun)$] \label{VSC^3} 
Let $\Phi$ be an index function and $\Rpen$ a proper, convex, lower-semicontinuous functional on $\X$. We say that $\fdagger\in \X$ satisfies 
the third order variational source condition $\VSC^3(\fdagger,\Phi,\Rpen,\Sfun)$ if 
there exist $\pbar\in\Y^*$ and 
	$\wbar\in \X$ such that $\T^*\pbar\in\partial\Rpen(\fdagger)$ and $\T\wbar\in\partial\Sfun^*(\pbar)$  
and if there exist constants $\beta\ge 0$, $\mu>1$ and $\overline{t}>0$ 
as well as $\f_t^*\in\partial \Rpen(\fdagger-t\wbar)$ for all $0<t \leq \overline{t}$ such that
	\begin{alignat*}{2}
	&\forall \f\in\X\; \forall t\in (0,\overline{t}] \colon &&\\ 
	&\lsp  \f_t^*-\T^*\pbar,\fdagger-t\wbar-\f\rsp \le && \breg{\Rpen}^{f_t^*}\!\paren{\f,\fdagger-t\wbar} \\
	& &&+ t^2\Phi\left(t^{-q}\norm{\T \f-\gdagger+t\T\wbar}^q\right)
		+\beta t^{2\mu}.
	\end{alignat*}
\end{defi}

\begin{rem}
{\new 
To see how $\VSC^2$ and $\VSC^3$ relate to other source conditions, recall from the introduction 
that the strongest first order variational source condition is $\VSC^1(\fdagger,C\sqrt{\cdot},\Rpen,\Sfun)$, which 
is equivalent to the existence of $\pbar\in\Y^*$ such that $\T^*\pbar\in\partial\Rpen(\fdagger)$ 
(see \cite[Propositions 3.35, 3.38]{Scherzer2008}). So by assuming the existence of such 
$\pbar\in\Y^*$, $\VSC^2$ and $\VSC^3$ are stronger than $\VSC^1$. Similarly, as discussed in the introduction 
they are also stronger than the multiplicative variational source conditions in \cite{Andreev2015,KH:10} 
and approximate (variational) source conditions (\cite{flemming:12b}).  

 Now 
let $\X$ and $\Y$ be Hilbert spaces and $\Rpen_{\rm sq}(\f):=\tfrac{1}{2}\|\f\|_{\X}^2$, $\Sfun_{\rm{sq}}(g)=\frac{1}{2}\norm{g}_\Y^2$. 
Then clearly the $\VSC^2(\fdagger,\Phi,\Rpen_{\rm sq},\Sfun_{\rm sq})$ is equivalent to $\VSC^2(\fdagger,\Phi)$.}
We also have that the $\VSC^3(\fdagger,\Phi,\Rpen_{\rm sq},\Sfun_{\rm sq})$ is equivalent to $\VSC^3(\fdagger,\Phi)$: In fact, for arbitrary $\beta\ge 0$ and $\mu>1$ 
the condition $\VSC^3(\fdagger,\Phi)$ is  equivalent to
	\begin{align*}
\forall \f\in\HilbertX \;\forall t>0\colon\qquad 
		\lsp  \wbar,\f \rsp\le\frac{1}{2}\norm{\f}^2+\Phi\left(\norm{\T\f}^2\right)+\beta t^{2\mu-2},
	\end{align*}
as the limit $t\to 0$ gives back the original inequality.
	Now we replace $\f$ by $\frac{\f-\fdagger+t\wbar}{t}$ and multiply by $t^2$ to see that this is equivalent to
	\begin{align*}
		\lsp -t\wbar, \fdagger-t\wbar-f \rsp \le \frac{1}{2}\norm{f-\fdagger+t\wbar}^2+t^2\Phi\left(\frac{\norm{\T\f-\gdagger+t\T\wbar}^2}{t^2} \right)+\beta t^{2\mu},
	\end{align*}
	 which is equivalent to $\VSC^3(\fdagger,\Phi,\Rpen_{\rm sq},\Sfun_{\rm sq})$.
\end{rem}

We can now state the main result of this section: 
\begin{thm} \label{rate_theorem}
Suppose Assumption \ref{assump_RS} and  that $\VSC^3(\fdagger,\Phi,\Rpen,\Sfun)$ is satisfied with constants $\beta,\mu$, and $\overline{t}$  
and that {\new $c^{-1}\delta\le \alpha^{q^*-1}\leq \overline{t}$  for some $c>0$}. Define $\widetilde{\Phi}(s)=\Phi(s^{q/r})$. Then the error is bounded by 
	\begin{align*}
		\breg{\Rpen}\!\paren{\itfal{2},\fdagger}\le C\left(\frac{\delta^q}{\alpha}
		+\alpha^{2(q^*-1)}\left(-\widetilde{\Phi}\right)^*\left(\frac{\tilde{C}\left(c+\norm{\T\wbar}\right)^{q-\conv}}{-\alpha^{q^*-1}}\right)
		+\beta\alpha^{2\mu(q^*-1)}\right)
	\end{align*}
with constants $C,\tilde{C}>0$ depending  {\new at most on $q$, $\conv$, $c$,} and $\cXR$ and $\csXR$ from Lemma \ref{xu-roach}.
\end{thm}

The proof consists of the following three lemmas. 
First we show that $\breg{\Rpen}(\itfal{2},\fdagger)$ is related to the Bregman distance 
$\frac{1}{\alpha}\breg{\Sfun^*}(-\alpha\pal,-\alpha\pbar)$ as we will later actually use
$\VSC^3(\fdagger,\Phi,\Rpen,\Sfun)$ to prove convergence rates for $\pal$.
\begin{lem} \label{starting_lemma}
	If $\T^*\pbar\in\partial\Rpen(\fdagger)$, then 
	\begin{align*}
		\breg{\Rpen}\!\paren{\itfal{2},\fdagger}\le\frac{2}{\alpha}\left(\Sfun\paren{\gdagger-\gobs}
		+\frac{1}{\csXR}\breg{\Sfun^*}\paren{-\alpha\pal,-\alpha\pbar}\right).
	\end{align*}
\end{lem}
\begin{proof}
	We apply Lemma \ref{Lemma11.19} with $\f=\fdagger$ to find
	\begin{align*}
	\breg{\Rpen}\!\paren{\itfal{2},\fdagger}\le\frac{1}{\alpha}\Sfun\paren{\gdagger-\gobs}
+\lsp \pbar-\pal,\gdagger-\gobs\rsp+\frac{1}{\alpha}\Sfun^*\paren{-\alpha(\pbar-\pal)}.
	\end{align*}
	 The generalized Young inequality applied to the middle term yields 
	\begin{align*}
\breg{\Rpen}\!\paren{\itfal{2},\fdagger} &\le\frac{2}{\alpha}\left(\Sfun\paren{\gdagger-\gobs}
  +\Sfun^*\paren{-\alpha(\pbar-\pal)}\right).
	\end{align*}
	As $\Y$ is $q$-smooth, $\Y^*$ is $q^*$ convex, so we can apply Lemma~\ref{xu-roach} to obtain 
	\begin{align*}
		\breg{\Rpen}\!\paren{\itfal{2},\fdagger}&\le\frac{2}{\alpha}\left(\Sfun\paren{\gdagger-\gobs}
  +\csXR^{-1}\breg{\Sfun^*}\paren{-\alpha\pal,-\alpha\pbar}\right).\qquad\qed
	\end{align*}
\end{proof}

The next lemma shows convergence rates in the image space. Such rates have also been shown under a first order variational source condition on $\fdagger$ in \cite[Theorem 2.3]{HW:14}.

\begin{lem} \label{img_conv}
Suppose there exist $\pbar\in \Y^*$ and $\wbar\in \X$ such that $\T^*\pbar\in\partial\Rpen(\fdagger)$  and $\T\wbar\in\partial\Sfun^*(\pbar)$. 
Then there exists a constant $C_q>0$ depending only on $q$ such that
	\begin{align*}
		\norm{\T\fal-\gdagger}\le C_q\left(\delta +\alpha^{q^*-1}\norm{\T\wbar}\right).
	\end{align*}
\end{lem}

\begin{proof}
	From \cite[Lemma 3.20]{Scherzer2008} we get 
	\begin{align*}
		\frac{1}{2^{q-1}q}\norm{\T\fal-\gdagger}^q \le \Sfun(\T\fal-\gobs)+\Sfun\paren{\gdagger-\gobs}.
	\end{align*}
	By the minimizing property of $\fal$ \eqref{eq:minimizing_property} we have
	\begin{align*}
\Sfun(\T\fal-\gobs)-\Sfun\paren{\gdagger-\gobs} &\le \alpha\left(\Rpen\paren{\fdagger}-\Rpen(\fal)\right) \\
	&=-\alpha\breg{\Rpen}\!\paren{\fal,\fdagger}-\alpha\lsp \T^*\pbar, \fal-\fdagger \rsp.
	\end{align*}
	Now using the non-negativity of the Bregman distance we have
	\begin{align*}
		\frac{1}{2^{q-1}q}\norm{\T\fal-\gdagger}^q
		&\le 2\Sfun\paren{\gdagger-\gobs}+\alpha \norm{ \pbar}\norm{ \T\fal-\gdagger} \\
		&\le \frac{2\delta^q}{q}+\frac{2^{-q}}{q}\norm{\T\fal-\gdagger}^q+\frac{(2\alpha)^{q^*}}{q^*}\norm{\pbar}^{q^*},
	\end{align*}
	where the last inequality follows from the $\norm{\gdagger-\gobs}\le \delta$ as well as from the generalized Young inequality. Therefore we have
	\begin{align*}
		\norm{\T\fal-\gdagger}^q \le 2^{q}q\left(\frac{2\delta^q}{q}+\frac{(2\alpha)^{q^*}}{q^*}\norm{\pbar}^{q^*}\right).
	\end{align*}
	The claim then follows from taking the $q$-th root and noticing that we have 
	$\norm{\pbar}^{q^*}=\norm{J_{q,\Y}(\T\wbar)}^{q^*}=\norm{\T\wbar}^q$ (see \eqref{eq:defi_duality_map}) 
	as well as $\alpha^{\frac{q^*}{q}}=\alpha^{q^*-1}$.
\qed\end{proof}

The main part of the proof of Theorem~\ref{rate_theorem} consists in the derivation of convergence rates 
for the dual problem:
\begin{lem} \label{dual_conv}
	Suppose that Assumption \ref{assump_RS} holds true and define $\alpha_q:=\alpha^{q^*-1}$, $\widetilde{\Phi}(s)=\Phi(s^{q/r})$.
		Moreover, let $\VSC^3(\fdagger,\Phi,\Rpen,\Sfun)$ hold true with constants $\beta,\mu$, and $\overline{t}$. 
		 If $\alpha$ is chosen such that 
		{\new $c^{-1}\delta\le \alpha_q\leq \overline{t}$, for some $c>0$,} then
		\begin{align*}
			\frac{1}{2\alpha}\breg{\Sfun^*}\paren{-\alpha\pal,-\alpha\pbar}\le C\frac{\delta^q}{\alpha}
			+\alpha_q^{2}\left(-\widetilde{\Phi}\right)^*\left(\frac{-\tilde{C}\left(c+\norm{\T\wbar}\right)^{q-r}}{\alpha_q}\right)+\beta\alpha_q^{2\mu},
		\end{align*}
		where $C,\tilde{C}>0$ depend at most on $q$, $\conv$, $c$, $\cXR$, and $\csXR$.
\end{lem}

\begin{proof}
It follows from \eqref{homogeneity} and $\T\wbar\in\partial\Sfun^*(\pbar)$ that 
$-\alpha_q\T\wbar\in\partial\Sfun^*(-\alpha\pbar)$. Together with \eqref{IA} we obtain 
	\begin{align*}
		\frac{1}{\alpha}\symbreg{\Sfun^*}\paren{-\alpha\pal,-\alpha\pbar}&=\frac{1}{\alpha}\lsp -\alpha\pbar- (-\alpha\pal), -\alpha_q\T\wbar- (\T\fal-\gobs)\rsp \\
		=&\lsp \T^*\pal-\T^*\pbar, \fdagger-\alpha_q\wbar-\fal\rsp+\lsp\pal-\pbar,\gobs-\gdagger\rsp.
	\end{align*}
	The second term $E:=\lsp\pal-\pbar,\gobs-\gdagger\rsp$ will be estimated later. 
	Artificially adding zero in the form $\f_{\alpha_q}^*-f_{\alpha_q}^*$ with $\f_{\alpha_q}^*\in \partial\Rpen(\fdagger-\alpha_q\wbar)$, we find
	\begin{align*} \begin{split}
		\frac{1}{\alpha}\symbreg{\Sfun^*}\paren{-\alpha\pal,-\alpha\pbar}&= 
		\lsp f_{\alpha_q}^*-\T^*\pbar, \fdagger-\alpha_q\wbar-\fal\rsp +E \\
		&\quad +\lsp \T^*\pal-f_{\alpha_q}^*, \fdagger-\alpha_q\wbar-\fal\rsp.
		\end{split}
	\end{align*}
In view of \eqref{IA} the last term is the negative symmetric Bregman distance $-\symbreg{\Rpen}(\fal, \fdagger-\alpha_q\wbar)$. The first term can be bounded using $\VSC^3(\fdagger,\Phi,\Rpen,\Sfun)$ 
	by choosing $\f=\fal$ and $t=\alpha_q$: 
	\begin{align*}
	\begin{split}
		\frac{1}{\alpha}\symbreg{\Sfun^*}\paren{-\alpha\pal,-\alpha\pbar}&\le  \alpha_q^{2}\widetilde{\Phi}\left(\alpha_q^{-\conv}\norm{\T\fal-\gdagger+\alpha_q\T\wbar}^r\right)+\beta\alpha_q^{2\mu}\\
		&\quad+ \breg{\Rpen}\paren{\fal,\fdagger-\alpha_q\wbar} -\symbreg{\Rpen}(\fal, \fdagger-\alpha_q\wbar)+E \\
		&\le \alpha_q^2\widetilde{\Phi}\left(\alpha_q^{-\conv}\norm{\T\fal-\gdagger+\alpha_q\T\wbar}^r\right)+\beta \alpha_q^{2\mu} +E.
		\end{split}
	\end{align*}
	Now we use our joker. We subtract \[\frac{1}{\alpha}\breg{\Sfun^*}(-\alpha\pbar,-\alpha\pal)=\frac{1}{\alpha}\breg{\Sfun}(\T\fal-\gobs,-\alpha_q T\wbar)\] (see \eqref{eq:bregman_identity}) from both sides leading to
	\begin{align}\label{eq:timmy}
	\begin{split}
		\frac{1}{\alpha}\breg{\Sfun^*}\paren{-\alpha\pal,-\alpha\pbar}
		&\le \alpha_q^2 \widetilde{\Phi}\left(\alpha_q^{-\conv}\norm{\T\fal-\gdagger+\alpha_q\T\wbar}^r\right) \\
		& \quad-\breg{\Sfun}(\T\fal-\gobs,-\alpha_q T\wbar) +\beta \alpha_q^{2\mu} +E.
		\end{split}
	\end{align}
	So we need to bound $\Delta:=\breg{\Sfun}(\T\fal-\gobs,-\alpha_q T\wbar)$ from below. By Lemma \ref{xu-roach} we have (as $q\leq \conv$) 
	\begin{align*}
		\Delta 	\ge \cXR\max\paren{\norm{\alpha_q\T\wbar},\norm{\T\fal-\gobs+\alpha_q\T\wbar}}^{q-r}\norm{\T\fal-\gobs+\alpha_q\T\wbar}^r.
	\end{align*}
	Moreover, it follows from Lemma \ref{img_conv} and the choice {\new $\delta\le c\alpha_q$} that 
	\begin{align*}
		\norm{\T\fal-\gobs+\alpha_q\T\wbar} &\le \norm{\T\fal-\gdagger}+\norm{\gdagger-\gobs}+\norm{\alpha_q\T\wbar} \\
		&\le C_q\left(\delta +\alpha_q\norm{\T\wbar}\right)\le C_q\alpha_q\left(c+\norm{\T\wbar}\right).
	\end{align*}
Therefore, 
	\begin{align*}
		\max\paren{\norm{\alpha_q\T\wbar},\norm{\T\fal-\gobs+\alpha_q\T\wbar}} 
		&\le \alpha_q\max\paren{\norm{\T\wbar},C_q\paren{c+\norm{\T\wbar}}} \\
		&\le \alpha_q\max\paren{1,C_q}\paren{c+\norm{\T\wbar}}.
	\end{align*}
	Hence, there exists a constant $\tilde{C}>0$ depending on $q$, $\cXR$, and $\conv$ such that
	\begin{align*}
		\frac{1}{\alpha}\Delta \ge 2^{r-1}\tilde{C}\paren{c+\norm{\T\wbar}}^{q-r}\alpha^{(q^*-1)(q-r)-1}\norm{\T\fal-\gobs+\alpha_q\T\wbar}^r.
	\end{align*}
Note that $(q^*-1)(q-r)-1=-(r-1)(q^*-1)$. In order to replace $\gobs$ by $\gdagger$ on the right hand side 
we use the inequality
	\begin{align*}
2^{1-r}\norm{\T\fal-\gdagger+\alpha_q\T\wbar}^r-\norm{\T\fal-\gobs+\alpha_q\T\wbar}^r\le \norm{\gdagger-\gobs}^r
	\end{align*}
	(see \cite[Lemma 3.20]{Scherzer2008}) leading to
	\begin{align*}
		-\frac{1}{\alpha}\Delta
		&\le \frac{\tilde{C}\paren{c+\norm{\T\wbar}}^{q-r}}{\alpha_q^{r-1}}
		\paren{-\norm{\T\fal-\gdagger+\alpha_q\T\wbar}^r	+ 2^{r-1}\delta^r}.
	\end{align*}
	Inserting this into \eqref{eq:timmy} yields
		\begin{align*}
		\frac{1}{\alpha}\breg{\Sfun^*}\paren{-\alpha\pal,-\alpha\pbar}&\le \alpha_q^{2}\widetilde{\Phi}\left(\alpha_q^{-\conv}\norm{\T\fal-\gdagger+\alpha_q\T\wbar}^r\right) 
		 +\beta\alpha_q^{2\mu}+E \\
		& \quad -\frac{\tilde{C}\paren{c+\norm{\T\wbar}}^{q-r}}{\alpha_q^{r-1}}\paren{\norm{\T\fal-\gdagger+\alpha_q\T\wbar}^r -2^{r-1}\delta^r}
		  \notag  \\
		&\le \alpha_q^{2}\sup_{\tau\ge 0}\left[-\tilde{C}\paren{c+\norm{\T\wbar}}^{q-r}\alpha_q^{-1}\tau -\left(- \widetilde{\Phi}(\tau)\right)\right] \\
		& \quad +\tilde{C}\paren{c+\norm{\T\wbar}}^{q-r}2^{r-1}\frac{\delta^r}{\alpha_q^{r-1}}+E+\beta\alpha_q^{2\mu}. 
	\end{align*}
The supremum equals $\left(-\widetilde{\Phi}\right)^*\left(-\tilde{C}\left(c+\norm{\T\wbar}\right)^{q-r}\alpha_q^{-1}\right)$, by the definition of the convex conjugate . 
	To deal with $E$ we use the generalized Young inequality
	\begin{multline*}
		\frac{1}{\alpha}\lsp\paren{\frac{\csXR q^*}{2}}^{\frac{1}{q^*}}(\alpha\pal-\alpha\pbar),\paren{\frac{\csXR q^*}{2}}^{\frac{-1}{q^*}}(\gobs-\gdagger)\rsp \\
		\le \frac{\csXR}{2\alpha}\norm{\alpha\pal-\alpha\pbar}^{q^*}		+\frac{1}{q}\paren{\frac{\csXR q^*}{2}}^{-\frac{q}{q^*}}\frac{\delta^q}{\alpha}
	\end{multline*}
	  and apply Lemma \ref{xu-roach}, using that $\Y^*$ is $q^*$ convex, to find 
	\begin{align*}
E &=		\lsp\pal-\pbar,\gobs-\gdagger\rsp\le \frac{1}{2\alpha}\breg{\Sfun^*}\paren{-\alpha\pal,-\alpha\pbar}+\frac{1}{q}\paren{\frac{\csXR q^*}{2}}^{-\frac{q}{q^*}}\frac{\delta^q}{\alpha}.
	\end{align*}
The assumption $\delta\le\alpha_q$, or equivalently $\delta^{r-q}\le\alpha_q^{r-q}$, implies 
	$\frac{\delta^r}{\alpha_q^{r-1}}\le\frac{\delta^q}{\alpha_q^{q-1}}=\frac{\delta^q}{\alpha}$.
Further $\paren{c+\norm{\T\wbar}}^{q-r}\le c^{q-r}$, hence there exists a constant $C>0$ depending on $q$,  $r$, $c$, $\cXR$, and $\csXR$ such that
	\begin{align*}
		\frac{1}{2\alpha}\breg{\Sfun^*}\paren{-\alpha\pal,-\alpha\pbar}&\le C\frac{\delta^q}{\alpha}+ 
		\alpha_q^{2}\left(-\widetilde{\Phi}\right)^*\left(-\tilde{C}\left(c+\norm{\T\wbar}\right)^{q-r}\alpha_q^{-1}\right) 
		+ \beta\alpha_q^{2\mu},
	\end{align*}
	 which completes the proof.
\qed\end{proof}

Now Theorem \ref{rate_theorem} is an immediate consequence of Lemma \ref{starting_lemma} and Lemma \ref{dual_conv}. 


\section{Verification of higher order variational source conditions}\label{sec:verification}
In this section we provide some examples how higher order VSCs can 
be verified for specific inverse problems.
 
\subsection{Hilbert spaces}\label{sec:verificationHilbert}
In the following we introduce spaces $\HilbertX_{\kappa}$ which are defined by 
conditions due to Neubauer \cite{Neubauer1997} and describe necessary and sufficient 
conditions for rates of convergence of spectral regularization methods. 
Let $E^{T^*T}_{\lambda}  := 1_{[0,\lambda)}(T^*T)$, $\lambda\geq 0$, denote 
the spectral projections for the operator $T^*T$ with the characteristic 
function $1_{[0,\lambda)}$ of the interval $[0,\lambda)$. For an index function 
$\kappa$ we define 
\begin{equation}\label{eq:decaySpace}
	\begin{aligned}
		\HilbertX^T_{\kappa}:=& \left\{\f \in\HilbertX \colon \norm{\f}_{\HilbertX_{\kappa}^T} < \infty
		\right\},
		\qquad 
		\norm{\f}_{\HilbertX_{\kappa}^T}:= \sup_{\lambda>0} \frac{1}{\kappa(\lambda)} 
		\norm{E^{T^*T}_{\lambda} \f}_{\HilbertX}.
	\end{aligned}
\end{equation}
The function $\kappa$ corresponds to the function in spectral source conditions 
$\fdagger\in \mathrm{ran}(\kappa(\T^*\T))$. The corresponding convergence rate 
function is 
\[
\Phi_{\kappa}(t):=\kappa\paren{\Theta_{\kappa}^{-1}(\sqrt{t})}^2,\qquad 
\Theta_{\kappa}(\lambda):= \sqrt{\lambda}\kappa(\lambda).
\]

In particular, $\Phi_{\id^{\nu/2}} = \id^{\nu/(\nu+1)}$. 
The following theorem is a generalization of \cite[Thm. 3.1]{HW:17} from the 
special case $l=1$ to general $l\in\N$:

\begin{thm}\label{thm:interp_vscl}
Let  $\kappa$ be an index function such that  $t\mapsto\kappa(t)^2/t^{1-\mu}$ 
is decreasing for some $\mu\in(0,1)$, $\kappa\cdot\kappa$	is concave, 
and $\kappa$ is decaying sufficiently rapidly such that 
\begin{equation}\label{eq:kappa_decay}
C_{\kappa}:=\sup_{0<\lambda\leq \|T^*T\|}
\frac{\sum_{k=0}^\infty\kappa(2^{-k}\lambda)^2}{\kappa(\lambda)^2} <\infty.
\end{equation}
Moreover, let $l\in\N$ and define $\tilde{\kappa}(t)=\kappa(t)t^{l/2}$. 
Then 
\begin{align}\label{eq:interp_vscl}
\fdagger\in \HilbertX_{\tilde{\kappa}}^{\T}
\quad \Leftrightarrow \quad 
\exists A>0\,:\,\VSC^{l\new +1}(\fdagger,A\Phi_{\kappa}).
\end{align}
\end{thm}

Note that condition \eqref{eq:kappa_decay} holds true for all power functions 
$\kappa(t) = t^{\nu}$ with $\nu>0$, but not for logarithmic functions 
$\kappa(t) = (-\ln t)^{-p}$ with $p>0$. The first two conditions on the other 
hand imply that $\kappa$ must not decay to $0$ too rapidly. They are both 
satisfied for power functions $\kappa(t) = t^{\nu/2}$ if and only if $\nu\in (0,1)$. 
We point out that for the case $l=1$ the condition \eqref{eq:kappa_decay} is not 
required. 

\begin{proof}
We first show for all $l\in\N$ that
\begin{align} 	 \label{evenn}
	 	\fdagger\in\HilbertX_{\tilde{\kappa}}^T 
		&\qquad \Leftrightarrow \qquad 
		\exists \, \itwbar{\frac{l}{2}}\in\HilbertX_{\kappa}^{T}\,:\, 
		\fdagger=(T^*T)^\frac{l}{2}\itwbar{\frac{l}{2}}
		\end{align}
which together with the special case $l={\new 0}$ from \cite[Thm 3.1]{HW:17} 
already implies \eqref{eq:interp_vscl} for even $l$:
	\begin{enumerate}[label=(\roman*)]
		\item Assume there exists $\itwbar{\frac{l}{2}}\in\HilbertX_{\kappa}^{T}$  
		such that $\fdagger=(T^*T)^\frac{l}{2}\itwbar{\frac{l}{2}}$. Define $\int_0^{\lambda+}:=\lim_{\varepsilon\searrow 0}\int_0^{\lambda+\varepsilon}$. Then,  
	\begin{align*}
			\norm{\fdagger}_{\HilbertX_{\tilde{\kappa}}^T}^2
			&=\sup_{\lambda>0}\frac{1}{\tilde{\kappa}(\lambda)^2}\norm{E_\lambda^{T^*T}(T^*T)^\frac{l}{2}\itwbar{\frac{l}{2}}}^2 \\
			&=\sup_{\lambda>0}\frac{1}{\tilde{\kappa}(\lambda)^2} \int_0^{\lambda+} \tilde{\lambda}^l \diff \norm{E_{\tilde{\lambda}}\itwbar{\frac{l}{2}}}^2 \\
			&\le \sup_{\lambda>0}\frac{1}{\tilde{\kappa}(\lambda)^2} 
			\int_0^{\lambda+} \lambda^l \diff \norm{E_{\tilde{\lambda}}\itwbar{\frac{l}{2}}}^2 \\
			&=\sup_{\lambda>0}\frac{1}{\kappa(\lambda)^2} 
	\int_0^{\lambda+} \diff \norm{E_{\tilde{\lambda}}\itwbar{\frac{l}{2}}}^2 
			=\norm{\itwbar{\frac{l}{2}}}_{\HilbertX_{\kappa}^{T}}^2<\infty.
		\end{align*}
		\item Now assume that $\fdagger\in\HilbertX_{\tilde{\kappa}}^T$. 
		It follows that 
\allowdisplaybreaks
		\begin{align*}
		\frac{1}{\kappa(\lambda)^2} \int_0^{\lambda+} \tilde{\lambda}^{-l} \diff \norm{E_{\tilde{\lambda}}\fdagger}^2
			&=\frac{1}{\kappa(\lambda)^2} \sum_{k=0}^\infty 
			\int_{2^{-k-1}\lambda}^{2^{-k}\lambda+} \tilde{\lambda}^{-l} \diff \norm{E_{\tilde{\lambda}}\fdagger}^2 \\
			&\le \frac{1}{\kappa(\lambda)^2} \sum_{k=0}^\infty 
			\int_{2^{-k-1}\lambda}^{2^{-k}\lambda+} (2^{-k-1}\lambda)^{-l} \diff \norm{E_{\tilde{\lambda}}\fdagger}^2\\
			&\le\frac{1}{\kappa(\lambda)^2} \sum_{k=0}^\infty \frac{\kappa(2^{-k}\lambda)^2}{\kappa(2^{-k}\lambda)^2(2^{-k-1}\lambda)^{\new l}}
			\!\int_0^{2^{-k}\lambda+}  \!\!\!\!\!\!\!\diff \norm{E_{\tilde{\lambda}}\fdagger}^2 \\
			&{\newnew =\lim_{\varepsilon\searrow 0}\frac{2^l}{\kappa(\lambda)^2} \sum_{k=0}^\infty \frac{\kappa(2^{-k}\lambda)^2}{\tilde{\kappa}(2^{-k}\lambda)^2}
			 \norm{E_{2^{-k}\lambda+\varepsilon}\fdagger}^2} \\
			&\le \frac{2^l}{\kappa(\lambda)^2}\sum_{k=0}^\infty \kappa(2^{-k}\lambda)^2 \norm{\fdagger}_{\HilbertX_{\tilde{\kappa}}^T}^2 
			\leq 2^lC_{\kappa}\norm{\fdagger}_{\HilbertX_{\tilde{\kappa}}^T}^2.
		\end{align*}
	This shows that 
	{\new $\itwbar{l/2}:=  \int_0^{\infty} \tilde{\lambda}^{-l/2} \diff E_{\tilde{\lambda}}\fdagger$ is well defined.}  	Moreover, we have 
	$\itwbar{l/2}{=\new (T^*T)^{-(l/2)}}\fdagger$ and $\norm{\itwbar{l/2}}_{\HilbertX_{\kappa}^T} <\infty$. 
	\end{enumerate}
To prove the theorem in the case of odd $l$ we use the polar decomposition 
$T= U(T^*T)^{1/2}$ with a partial isometry $U$ satisfying $N(U)=N(T)$ 
and set $\itpbar{\frac{l+1}{2}}:=U\itwbar{\new \frac{l}{2}}$. 
As $U:\HilbertX_{\kappa}^{T}\to\HilbertY_{\kappa}^{T^*}$ is an isometry, 
\eqref{evenn} implies
\begin{align}
		\label{oddn}
	 	\fdagger\in\HilbertX_{\tilde{\kappa}}^T 
		&\qquad \Leftrightarrow \qquad 
		\exists \, \itpbar{\frac{l+1}{2}}\in\HilbertY_{\kappa}^{T^*} \,:\,
		\fdagger=(T^*T)^\frac{l-1}{2}T^*\itpbar{\frac{l+1}{2}}.
	 \end{align}
Applying \eqref{eq:interp_vscl} for $l={\new 0}$ from \cite[Thm. 3.1]{HW:17} 
to $\HilbertY$ and $\T\T^*$ yields \eqref{eq:interp_vscl} for the case of odd $l$. 
\qed\end{proof}

{\new The equivalence \eqref{eq:interp_vscl} together with the equivalence in \cite[Prop. 4.1]{Albani2016} 
also shows that in Hilbert spaces higher order variational source conditions are equivalent 
to certain symmetrized multiplicative variational source conditions.} 
We have already seen at the end of \S~\ref{sec:Hilbert_spaces} that 
$\VSC^l(\fdagger,A\id^{\nu/(\nu+1)})$ implies the order optimal convergence rate 
$\norm{\itfal{m}-\fdagger}= \landauO{\delta^{(l-1+\nu)/(l+\nu)}}$ for an optimal 
choice of $\alpha$ and $m\geq l/2$. 
It follows from \cite{HW:17} and Theorem \ref{thm:interp_vscl} 
that $\VSC^l\left(\fdagger,A\id^{\nu/(\nu+1)}\right)$, with $\nu\in (0,1)$ is not only a sufficient condition 
for this rate of convergence, but in contrast to spectral H\"older source 
conditions also a necessary condition: 

\begin{cor}
Let $l\in\N$, $m\geq l/2$, and $\nu\in (0,1)$. Moreover, let $\fdagger\neq 0$ 
and let $\itfal{m}=\itfal{m}(\gobs)$ denote the $m$-times iterated Tikhonov 
estimator. Then the following statements are equivalent:
\begin{enumerate}
\item 
\[\exists A>0\,:\,\VSC^l\paren{\fdagger,A\id^{\nu/(\nu+1)}}\]
\item
\[\exists C>0\;\forall \delta>0\;:\,
\sup_{\delta>0}\inf_{\alpha>0}\sup_{\|\gobs-\T\fdagger\|\leq \delta}
\norm{\itfal{m}(\gobs)-\fdagger}\leq  C\delta^{(l-1+\nu)/(l+\nu)}\]
\end{enumerate}
\end{cor} 

For operators which are $a$-times smoothing in the sense specified below, 
higher order variational  source conditions can be characterized in terms 
of Besov spaces in analogy to first order variational source conditions 
(see \cite{HW:17}): 

\begin{cor}\label{cor:besov_char}
Assume that $\manifold$ is a connected, smooth Riemannian manifold, which is 
complete, has injectivity radius $r>0$ and a bounded geometry (see \cite{triebel:92} 
for further discussions) and that 
$\T:H^s(\manifold)\to H^{s+a}(\manifold)$ is bounded and boundedly invertible 
for some $a>0$ and all $s\in\R$. 
Then for all $\fdagger\in L^2(\manifold)$, all $l\in\N$ and all $\nu\in (0,1)$ 
we have 
\begin{subequations}\label{eqs:besov_char}
\begin{align}
\label{eq:besov_char}
&\exists A>0\,:\,\VSC^l\paren{\fdagger,A\id^{\frac{\nu}{\nu+1}}} 
&&\Leftrightarrow\quad 
\fdagger \in B^{(l-1+\nu)a}_{2,\infty}(\manifold),\\
\label{eq:besov_char_integer}
&\exists A>0\,:\,\VSC^l\paren{\fdagger,A\sqrt{\cdot}} 
\hspace*{-12ex}
&& \Leftrightarrow\quad 
\fdagger \in B^{la}_{2,2}(\manifold) = H^{la}(\manifold).
\end{align}
\end{subequations}
\end{cor}

\subsection{Non-quadratic smooth penalty terms}


The results in this subsection do not require Hilbert spaces and are formulated in the more general setting of Assumption \ref{assump_RS}.  In \cite[Lemma 5.3]{Grasmair2013} it is stated that under a smoothness condition on $\Rpen$ the $\VSC^2(\fdagger,\Phi,\Rpen,\Sfun)$ follows from the benchmark condition $\T^*\pbar\in\partial\Rpen(\fdagger)$, $\T\wbar\in\partial\Sfun^*(\pbar)$. In the following proposition we generalize this result to rates corresponding to 
H\"older conditions with exponent $\nu\in(1,2)$ using the technique of \cite[Thm.~2.1]{HW:17}. 

\begin{prop}[verification of $\VSC^2(\fdagger,\Phi,\Rpen,\Sfun)$]\label{prop_vsc2_proof}
	Suppose that Assumption \ref{assump_RS} holds true. {\new Let $\widetilde{\X}$ be a Banach space continuously 
	embedded in $ \X$  and assume that
$\Rpen$ is continuously Fr\'{e}chet-differentiable in a neighborhood of $\fdagger\in \widetilde{\X}$ with respect to 
$\norm{\cdot}_{\widetilde{\X}}$ and {\newnew $\Rpen':\widetilde{\X}\to\widetilde{\X}^*$} is uniformly Lipschitz continuous with respect to 
$\norm{\cdot}_{\widetilde{\X}}$ in this neighborhood.
Further assume that there exists $\pbar\in\Y^*$ such that $ T^*\pbar\in \partial\Rpen(\fdagger)$, $\Rpen'[\fdagger]=(T^*\pbar)|_{\widetilde{\X}}$
and define $\pstar:=J_{q^*,\Y^*}(\pbar)\in\Y$.} Suppose that 
	there exists a family of operators $P_k\in L(\Y)$ indexed by 
	$k\in \mathbb{N}$ such that {\new $P_k\pstar\in T\widetilde{\X}$} for all $k\in\mathbb{N}$, 
and let
	\begin{align}\label{eq:proj_for_pstar}
	&\kappa_k:=\|(I-P_k)\pstar\|_{\Y},
	\qquad\qquad \sigma_k:=\max\left\{\norm{T^{-1}P_k\pstar}_{\widetilde{\X}},1\right\}.
	\end{align}
If $\lim_{k\to\infty}\kappa_k = 0$, then 
there exists $C>0$ such that $\VSC^2(\fdagger,\Phi,\Rpen,\Sfun)$ holds true with the index function 
	\begin{align}
	\Phi(\tau):= C \inf_{k\in \mathbb{N}}\left[\sigma_k\tau^{1/2} \label{eq:VSC^2_index_function}
        +  \kappa_k^{\q}\right].
	\end{align}
\end{prop}

\begin{proof}
	We show \eqref{eq:VSC^2_inequality} for all $\p\in\Y^*$ by distinguishing three cases: \\
	\emph{Case 1:} $\p\in \mathcal{A}:=\{\p\in \Y^* : \lsp \pbar-p,\pstar\rsp \le \frac{1}{A}	\breg{\Rpen^*}\!\paren{T^*p,T^*\pbar}^\frac{1}{2} \}$, with a constant $A>0$ whose exact value will be chosen later. 
	For these $\p$ the inequality thus holds with $\Phi(\tau)=\frac{1}{A}\sqrt{\tau}$ which is smaller than \eqref{eq:VSC^2_index_function} for $C\ge 1/A$. \\
	\emph{Case 2:} $\p\in \mathcal{B}:=\{\p\in \Y^* : \breg{\Sfun^*}\!\paren{p,\pbar}^{1-1/q^*}  \ge 2c_{q^*,Y^*}^{-1}\norm{\pstar} \}$. $\Y^*$ is $q^*$-convex, as $\Y$ is $q$-smooth, so by Lemma \ref{xu-roach}  we have for all $\p\in \mathcal{B}$ that
	\begin{align*}
		\lsp \pbar-p,\pstar\rsp \le \norm{\pbar-p}\norm{\pstar}\le \frac{1}{2}\breg{\Sfun^*}\!\paren{p,\pbar},
	\end{align*}
	so \eqref{eq:VSC^2_inequality} also holds for $\p\in \mathcal{B}$. \\
	\emph{Case 3:} $\p\in  \Y^*\setminus(\mathcal{A}\cup \mathcal{B})$. 
{\new By our regularity assumptions on $\Rpen$, there exist constants $C_{\fdagger},c>0$ such that 
 for 
all $\f\in\widetilde{\X}$ with $\norm{\f-\fdagger}_{\widetilde{\X}}\le C_{\fdagger}$ we have the first order Taylor approximation
	\begin{align*}
		\Rpen(f)\le \Rpen(\fdagger)+\lsp T^*\pbar,f-\fdagger\rsp +\frac{c}{2}\norm{f-\fdagger}_{\widetilde{\X}}^2,
	\end{align*}	
	where $c$ is the Lipschitz constant of $\Rpen'$}. Applying Young's inequalities 
	$\Rpen(\f)+\Rpen^*(\T^*p) \geq \lsp \T^*p,\f\rsp$ and 
	$\Rpen(\fdagger)+\Rpen^*(\T^*\pbar) = \lsp \T^*\pbar,\fdagger\rsp$, we find
	\begin{align*}
		\Rpen^*(T^*p)\ge \Rpen^*(T^*\pbar)+\lsp T^*(p-\pbar),\fdagger\rsp +\lsp T^*(p-\pbar),f-\fdagger\rsp-\frac{c}{2}\norm{f-\fdagger}_{\widetilde{\X}}^2
	\end{align*}
	for all $p\in \Y^*$ and for all $\f\in \widetilde{\X}$ with $\norm{\f-\fdagger}_{\widetilde{\X}}\le C_{\fdagger}$,
	which is equivalent to
	\begin{align}\label{eq:dual_product_bound}
		\lsp T^*(p-\pbar),f-\fdagger\rsp \le \breg{\Rpen^*}(T^*p,T^*\pbar) +\frac{c}{2}\norm{f-\fdagger}_{\widetilde{\X}}^2
	\end{align}
	for all $p\in \Y^*$ and for all $\f\in \widetilde{\X}$ with $\norm{\f-\fdagger}_{\widetilde{\X}}\le C_{\fdagger}$. We decompose 
	the left hand side of \eqref{eq:VSC^2_inequality} as follows:
	\begin{align*}
			\lsp \pbar-\p,\pstar	\rsp =	\lsp \pbar-\p,P_k\pstar \rsp + \lsp \pbar-\p,(I-P_k)\pstar \rsp .
	\end{align*}
	Now for some small $\varepsilon>0$ choose $\f$ in \eqref{eq:dual_product_bound} as $\f=\fdagger+\varepsilon T^{-1}P_k\pstar$. Then we can conclude that
	\begin{align*}
		\lsp \pbar-\p,P_k\pstar \rsp = \lsp T^*(\pbar-\p),T^{-1}P_k\pstar	\rsp\le \frac{1}{\varepsilon}\breg{\Rpen^*}(T^*p,T^*\pbar)+\frac{c\varepsilon}{2}\norm{T^{-1}P_k\pstar}^2_{\new \widetilde{\X}}.
\end{align*}
As $p\notin \mathcal{B}$ we know that $\norm{\pbar-p}$ is bounded, say $\norm{\pbar-p}\le B$. Now choose $A$ from above as $A=\frac{C_{\fdagger}}{B\norm{\pstar}}$. 
Then from $p\notin \mathcal{A}$ we know, that 
\begin{align*}
	\breg{\Rpen^*}(T^*p,T^*\pbar)^\frac{1}{2}\le A\norm{\pbar-p}\norm{\pstar}\le AB\norm{\pstar}\le C_{\fdagger}
\end{align*}
so we can choose $\varepsilon=\breg{\Rpen^*}(T^*p,T^*\pbar)^\frac{1}{2}/\sigma_k$, which ensures 
$\norm{f-\fdagger}\le C_{\fdagger}$. Therefore we have
	\begin{align*}
		\lsp \pbar-\p,P_k\pstar \rsp\le \paren{1+\frac{c}{2}}\sigma_k\breg{\Rpen^*}(T^*p,T^*\pbar)^\frac{1}{2}.
	\end{align*}	
	Combining everything and using Lemma \ref{xu-roach} with $\Y^*$ being $q^*$-convex we find
	\begin{align*}
		\lsp \pbar-\p,\pstar\rsp &\le \paren{1+\frac{c}{2}}\sigma_k\breg{\Rpen^*}(T^*p,T^*\pbar)^\frac{1}{2} + \kappa_k\norm{\pbar-\p} \\
		&\le \paren{1+\frac{c}{2}}\sigma_k\breg{\Rpen^*}(T^*p,T^*\pbar)^\frac{1}{2} +C_q\kappa_k^q+\frac{1}{2}\breg{\Sfun^*}\!\paren{p,\pbar}, \\
		&\le \frac{1}{2}\breg{\Sfun^*}\!\paren{p,\pbar} +C\paren{\sigma_k\breg{\Rpen^*}(T^*p,T^*\pbar)^\frac{1}{2} +\kappa_k^q},
	\end{align*}
	with $C_q=	\frac{1}{q}\paren{\frac{2}{q^*c_{q^*,\Y^*}}}^{q/q^*}$\!\!, $C=\max\left\{\frac{2+c}{2},C_q, \frac{1}{A} \right\}$, which completes the proof.
\qed\end{proof}

\begin{prop}[verification of $\VSC^3(\fdagger,\Phi,\Rpen,\Sfun)$] \label{prop_vsc3_proof}
Suppose that Assumption \ref{assump_RS} holds true {\new and let $\wbar\in \X$ as in Definition \ref{VSC^3} exist.
{ \newnew Let $0<\bar{t}\le 1$ and assume that for all  $f^*\in\partial\Rpen(\fdagger)$ there exists some $\omega^*\in\X^*$ such that
\begin{align}\label{eq:taylor_for_subdiff}
	\norm{f^*-f_t^*-t\omega^*}\le C_{\wbar}t^2,
\end{align}
whenever $0<t\le \bar{t}$ and $f_t^*\in\partial\Rpen(\fdagger-t\wbar)$.} This last assumption follows for example from $\Rpen$ being two times Fr\'{e}chet-differentiable in $\X$ in a neighborhood of $\fdagger$ with $\Rpen''\colon \X\to L(\X,\X^*)$  uniformly Lipschitz continuous in this neighborhood.} Further assume
	\begin{align}\label{eq:bregman_lower_bound}
	\breg{\Rpen}\!\paren{f_1,f_2}\ge C_\mu\norm{f_1-f_2}^\mu
	\end{align}
for some  $\mu>1$, $C_{\mu}>0$ and all $f_1,f_2\in \dom(\Rpen)$. We have:

\begin{enumerate}
	\item 
	If $\omega^* = \T^*\pbar^{(2)}$ for some $\pbar^{(2)}\in\Y^*$, then 
	$\VSC^3(\fdagger,\Phi,\Rpen,\Sfun)$ holds true with $\Phi(\tau):=\norm{\pbar^{(2)}}\tau^{1/q}$. 
	\item 
	Suppose that $\mu\leq 2$, $\tfrac{1}{\mu}+\tfrac{1}{\mu^*}=1$, 
 and that 
	there exists a family of operators $P_k\in L(\X^*)$ indexed by 
	$k\in \mathbb{N}$ such that $P_k\omega^*\in\T^*\Y^*$  for all $k\in \mathbb{N}$, and let
	\begin{align}\label{eq:proj_for_omegastar}
	&\kappa_k:=\|(I-P_k)\omega^*\|_{\X^*},
\qquad\qquad \sigma_k:= \|(\T^*)^{-1}P_k\omega^*\|_{\Y^{\new *}}.
	\end{align}
If $\lim_{j\to\infty}\kappa_k = 0$,
then $\VSC^3(\fdagger,\Phi,\Rpen,\Sfun)$ holds true with the index function 
	\begin{align}\label{eq:Phi_inf}
	\Phi(\tau):= \inf_{k\in \mathbb{N}}\left[\sigma_k\tau^{1/q}
        + \frac{\kappa_k^{\mu^*}}{\mu^*(C_\mu \mu)^{\mu^*/\mu}} \right].
	\end{align}
\end{enumerate}	
\end{prop}
\begin{proof}
{\new Firstly, to prove that \eqref{eq:taylor_for_subdiff} is implied by two times differentiability with $\Rpen''$ Lipschitz continuous, recall that 
$\partial \Rpen(f) = \{\Rpen'[f]\}$ if $\Rpen$ is Fr\'{e}chet-differentiable in $\X$, then 
by the first order Taylor approximation of $t\mapsto \Rpen'[\fdagger-t\wbar]$ at $t=0$ we have
\begin{align*}
	\norm{\Rpen'[\fdagger-t\wbar]-\Rpen'[\fdagger] + t\Rpen''[\fdagger](\wbar,\cdot)}_{\X^*} &\leq  C t^2\norm{\wbar}^2 
\end{align*}
for some $C>0$. Thus  \eqref{eq:taylor_for_subdiff} holds with $\omega^*=\Rpen''[\fdagger](\wbar,\cdot)$ and $C_{\wbar}=C\norm{\wbar}^2$.   
Now let \eqref{eq:taylor_for_subdiff} hold true, then we have for all $f_t^*\in\partial\Rpen(\fdagger-t\wbar)$ that
	\begin{align}\label{eq:aux_smooth_pen}
		\lsp f_t^*-\T^*\pbar, \fdagger-t\wbar-\f\rsp 
		&\le -t\lsp\omega^*, \fdagger-t\wbar-\f\rsp + C_{\wbar}t^2\norm{\fdagger-t\wbar-f}.
	\end{align}
Then using \eqref{eq:bregman_lower_bound} and Young's 
inequality, we find that 
\begin{align*}
 C_{\wbar}t^2\norm{\fdagger-t\wbar-f} 
&\leq \gamma t^2\breg{\Rpen}\!\paren{f,\fdagger-t\wbar}^\frac{1}{\mu}
\leq  \breg{\Rpen}\!\paren{f,\fdagger-t\wbar}+\beta t^{2\mu^*}
\end{align*}
with $\gamma:=C_{\wbar}C_\mu^{-\frac{1}{\mu}}$ and 
$\beta:=\frac{1}{\mu^*}\mu^{-\frac{\mu^*}{\mu}}\gamma^{\mu^*}$.}

So we only need to bound the first term on the right hand side of \eqref{eq:aux_smooth_pen} and this is done in two ways based on the two different assumptions:
\begin{enumerate}
\item If $\omega^* = \T^*\pbar^{(2)}$, then 
	\begin{align*}
		-t\lsp \omega^*, \fdagger-t\wbar-f\rsp 
		&=  -t\lsp \pbar^{(2)}, \gdagger-t\T\wbar-\T f\rsp \\
		&\le t^2\norm{\pbar^{(2)}}\left(t^{-q}
	\norm{\T f-\gdagger+t\T\wbar}^q\right)^{1/q}.
	\end{align*}
	Hence Assumption \ref{VSC^3} holds true $\Phi(\tau):=\norm{\pbar^{(2)}}\tau^{1/q}$. 
	\item In the second case we have for all $k\in \mathbb{N}$ with $c_\mu:=\frac{1}{\mu^*(C_\mu \mu)^{\mu^*/\mu}}$ that 
	\begin{align*}
	&-t\lsp \omega^*, \fdagger-t\wbar-f\rsp \\
	 &= -t \lsp P_k \omega^*, \fdagger-t\wbar-f\rsp -t\lsp (I-P_k)\omega^*, \fdagger-t\wbar-\f\rsp \\
	&\leq t\sigma_k \norm{\T \f-\gdagger+t\T\wbar} + t\kappa_k \norm{\fdagger-t\wbar-\f}\\
	 &\leq t^2 \left(	\sigma_k t^{-1}\norm{\T \f-\gdagger+t\T\wbar} + c_\mu
	t^{\mu^*-2}\kappa_k^{\mu^*}
	\right) 
	 + C_\mu\norm{\fdagger-t\wbar-\f}^\mu \\
&\leq t^2 \left(
	\sigma_k t^{-1}\norm{\T \f-\gdagger+t\T\wbar} +  c_\mu \kappa_k^{\mu^*}
	\right) + 	\breg{\Rpen}\!\paren{\f,\fdagger-t\wbar}
	\end{align*}
	for {\new $t\le \bar{t}\le 1$} as $\mu^*\geq 2$. 
	Substituting $\tau=t^{-q}\norm{\T f-\gdagger+t\T\wbar}^q$ and taking the infimum over $k$ shows 
	Assumption \ref{VSC^3} with $\Phi$ given by \eqref{eq:Phi_inf}. It follows as in 
\cite[Thm.~2.1]{HW:17} that $\Phi$ is an index function. \qed
\end{enumerate}
\end{proof}

\begin{exam}
{\newnew Let $\Omega$ be some measurable space with $\sigma$-finite measure.} Let $\Y=L^q(\Omega)$ for $1<q\le 2$ such that $\Y$  is $q$-smooth and $2$-convex. 
Moreover, let $\X=L^\mu(\Omega)$ with $\mu=2$ or {\new $\mu\ge 3$ such that the norm $\norm{\cdot}_\X$ is two times differentiable with Lipschitz continuous second derivative}
and choose $\Rpen(\cdot)=\frac{1}{\mu}\norm{\cdot}_\X^\mu$. Then \eqref{eq:bregman_lower_bound} holds true 
by Lemma \ref{xu-roach}. Assume that $\pbar$, $\wbar$, and  $\pbar^{(2)}$ are 
given as in Proposition \ref{prop_vsc3_proof}, Case 1.
Then Assumption \ref{VSC^3} holds true,  and Theorem \ref{rate_theorem} yields 
	 \begin{align}		
	 	\breg{\Rpen}\!\paren{\itfal{2},\fdagger}
		&=\mathcal{O}\left(\frac{\delta^q}{\alpha} + {\newnew \alpha_q^2(-\Phi})^*\left(-C\alpha_q^{-1}\right)+\alpha_q^{2\mu^*}\right),
\end{align}
where we again use the notation $\alpha_q:=\alpha^{q^*-1}$. 
Note that the Fenchel conjugate {\newnew $(-\Phi)^*$ fulfills $(-\Phi)^*(-s)\sim 1/s$} such that 
\begin{align}\label{awesome_bound}
\inf_{\alpha>0}\breg{\Rpen}\!\paren{\itfal{2},\fdagger}
&=\mathcal{O}\!\left(\inf_{\alpha>0}\left[\frac{\delta^q}{\alpha}+\alpha_q^{3}+\alpha_q^{2\mu^*}\right]\right)
	=\mathcal{O}\!\paren{\delta^{\frac{q\min(3,2\mu^*)}{q-1+\min(3,2\mu^*)}}}.
\end{align}
The best known error bound for Tikhonov regularization (requiring the existence of $\pbar$ and $\wbar$ as above) is 
	 \begin{align}\label{eq:old_bound}
	 	\inf_{\alpha>0}\breg{\Rpen}\!\paren{\fal,\fdagger}
		&=\mathcal{O}\!\left(\inf_{\alpha>0}\left[\frac{\delta^q}{\alpha}+\alpha_q^{2}\right]\right) 
		= \mathcal{O}\!\paren{\delta^{\frac{2q}{q+1}}}.
	 \end{align}
	(see \cite{Neubauer2010}). 
	As $\mu^*>1$ we see that the bound \eqref{awesome_bound} for $\itfal{2}$ is better than \eqref{eq:old_bound}. 
	In particular, if $\mu=2$ or {\new $\mu=3$} and hence $\min(3,2\mu^*)=3$, the right hand side in \eqref{awesome_bound} is $\mathcal{O}\left(\delta^{3q/(q+2)}\right)$, 
  which for $q=2$ yields the optimal bound $\|\itfal{2}-\fdagger\| = \mathcal{O}\left(\delta^{3/4}\right)$ for the spectral source 
	condition \eqref{eq:ssc} with $\nu=3$. However, for $\mu>3$ the rate \eqref{awesome_bound} is slower. 
\end{exam}

\subsection{Application to iterated maximum entropy regularization}\label{sec:ME}
In this subsection we will apply Propositions \ref{prop_vsc2_proof} and \ref{prop_vsc3_proof} to the case that 
the penalty term is chosen as a cross-entropy term given by the Kullback-Leibler divergence 
\begin{align}\label{eq:defi_KL}
\Rpen(\f) := \KL(\f,\f_0):= \int_{\manifold}\left[\f \ln \frac{\f}{\f_0}-\f+\f_0\right]\,dx
\end{align}
for some Riemannian manifold $\manifold$. 
Here $\f_0$ is some a-priori guess of $\f$, possibly constant. For more background information 
and references on entropy regularization we refer to \cite{skilling:89}. 
Under the source condition \eqref{eq:SSChalf} convergence rates 
of order $\|\hat{\f}_{\alpha}-\fdagger\|_{L^1}= \mathcal{O}(\sqrt{\delta})$ were shown 
in \cite{eggermont:93} by variational methods and 
in \cite{EL:93} by a reformulation as Tikhonov regularization with quadratic penalty term 
for a nonlinear forward operator. In \cite{resmerita:05} the faster rate 
$\|\hat{\f}_{\alpha}-\fdagger\|_{L^1}= \mathcal{O}(\delta^{2/3})$ was obtained under the source condition 
$T^*T\wbar\in \partial \Rpen(\fdagger)$. 

A simple computation shows that 
\[
\breg{\Rpen}^{\varphi^*}(\f,\varphi)= \KL(\f,\varphi)
\]
if $\varphi^*\in \partial \Rpen(\varphi)$, i.e.\ $\varphi^*= \ln (\varphi/\f_0)$. 
Let $\HilbertY$ be a Hilbert space and $\T\colon L^1\left( \manifold\right)\to  \HilbertY$ linear and bounded. 
We want to approximate $\fdagger\in \mathcal{C}$ where 
$\mathcal{C}\subset L^1(\manifold)$ is closed and convex, from noisy data 
$\gobs\in L^2$ with \[\norm{\gdagger-\gobs}_{\HilbertY}\le \delta\] and some a-priori 
guess $\f_0\in \mathcal{C}$ of $\fdagger$. The set $\mathcal{C}$ may contain 
only probability densities or further a-priori information such as 
box-constraints. 
To this end we apply generalized Tikhonov regularization in the form of maximum entropy regularization
	\[\hat{\f}_\alpha\in \argmin_{\f\in \mathcal{C}} \left[ \norm{T\f-\gobs}^2_{\HilbertY}+\alpha \KL\paren{\f,\f_0} \right] . \]
This amounts to choosing $\Rpen(\f):=\KL(\f,\f_0) + \iota_{\mathcal{C}}(\f)$ with the indicator function 
$\iota_{\mathcal{C}}(\f):=0$ if $\f\in\mathcal{C}$ and $\iota_{\mathcal{C}}(\f):=\infty$ else. 	
{\new If all iterates are in the interior of $\mathcal{C}$,} 
Bregman iteration is given by
\[\itfal{n}\in \argmin_{\f\in\mathcal{C}}\left[ \norm{T\f-\gobs}^2_{\HilbertY}+\alpha \KL\paren{\f,\itfal{n-1}} \right], 
\]	
{\new otherwise the iteration formula may involve an element of the normal cone of $\mathcal{C}$ at $\itfal{n-1}$.}
\begin{thm} \label{thm:ME_rates}
Let $\torus:=\R/\Z$, let $\manifold:=\mathbb{T}^d$ be the $d$-dimensional torus, $\Y=L^2(\mathbb{T}^d)$ and define $\Sfun_{\rm{sq}}(g)=\frac{1}{2}\norm{g}_\Y^2$. 
Suppose that $\T$ and its $L^2$-adjoint $T^*$ are $a>0$ times smoothing  
in the sense that 
$\T,\T^*:B^s_{p,q}(\torus^d)\to B^{s+a}_{p,q}(\torus^d)$ 
are isomorphisms for all $s\in [0,3a]$, $p\in [2,\infty]$, $q\in [2,\infty]$. 
Moreover, suppose there exists $\rho>0$ such that 
\begin{align}\label{eq:fdag_boundedness}
\rho \leq \frac{\fdagger}{f_0} \leq \rho^{-1}\qquad \mbox{a.e.\ in }\torus^d
\end{align}
and 
\begin{itemize}
	\item either $\mathcal{C} \subset \left\{\f\in L^1\colon \f\geq 0, \int \f\,dx =1\right\}$, then we set
	$p:=\infty$ 
	\item or $\sup_{f\in \mathcal{C}} \norm{f}_{L^{\infty}}<\infty$ and $a>d/2$, then we set $p:=2$. 
\end{itemize}
We make a further case distinction:
\begin{enumerate}
	\item Assume that
		\begin{align*}
			\frac{\fdagger}{f_0} \in B^s_{p,\infty}(\torus^d)\qquad 
			\mbox{for some }s\in \left(a,2a\right). 
		\end{align*}
		Then there exists $C>0$ such that $\VSC^2(\fdagger,\Phi,\KL(\cdot,\f_0),\Sfun_{\rm{sq}})$ holds true with 
		\[
			\Phi(\tau) = C \tau^{\frac{s-a}{s}}.
		\]
	\item Assume additionally to \eqref{eq:fdag_boundedness} that $\fdagger\ge \rho$ and
	\begin{align}\label{eq:fdag_smoothness}
\fdagger, \f_0\in B^s_{p,\infty}(\torus^d)\qquad 
\mbox{for some }s\in \left( 2a+\frac{d}{p},3a\right). 
\end{align}
Then there exists $C>0$ such that $\VSC^3(\fdagger,\Phi,\KL(\cdot,\f_0),\Sfun_{\rm{sq}})$ holds true with $\mu=2$ and 
\[
\Phi(\tau) = C \tau^{\frac{s-2a}{s-a}}.
\]
\end{enumerate}
In all four cases we obtain for the parameter choice 
$\alpha\sim \delta^{\frac{2a}{s +a}}$ the convergence rate 
\begin{align}\label{eq:ME_rate}
\KL\paren{\itfal{2},\fdagger}= \mathcal{O}\paren{\delta^{\frac{2s}{s+a}}},\qquad 
\delta\searrow 0.
\end{align}

\end{thm}

\begin{proof}
Note that due to \eqref{eq:fdag_boundedness} the functional $\Rpen$ is {\new Fr\'{e}chet differentiable 
at $\fdagger$ in $L^{\infty}(\torus^d)$ and 
\begin{align}
\begin{split} \label{eq:KL_derivatives}
\Rpen'[\fdagger](g) = \int_{\torus^d} \ln\paren{\frac{\fdagger}{f_0}} g\,dx, \quad 
&\Rpen''[\fdagger](g,h) = \int_{\torus^d} \frac{1}{\fdagger} h g\,dx \\
\mbox{and}\quad &\Rpen'''[\fdagger](g,h,h) = -\int_{\torus^d} \left(\frac{1}{\fdagger}\right)^2 h^2 g\,dx.
\end{split}
\end{align}
Hence, under the given regularity assumption we have
\[
T^*\pbar=\ln\left(\frac{\fdagger}{f_0}\right).
\] 
Local Lipschitz continuity of $\Rpen'$ w.r.t.\  $\widetilde{\X}:= L^{\infty}(\torus^d)$ follows from 
local boundedness of $\Rpen''$  w.r.t.\  $\widetilde{\X}$.  }
We are going to verify the assumptions of Propositions \ref{prop_vsc2_proof} and \ref{prop_vsc3_proof} choosing 
$\X = L^{\pconj}(\torus^d)$ with the conjugate exponent $\pconj$ of $p$. 
The operators $P_k$ for $k\in \mathbb{N}_0$ are chosen as quasi-interpolation operators onto 
level $k$ of a dyadic spline space of sufficiently high order as defined in \cite[eq.~(4.20)]{dVP:88}. 
Then we have the following two inequalities, where $C>0$ now and in the following will denote a generic constant. By \cite[Thm.~4.5]{dVP:88} we have for all $t>0$ and all $h\in B^t_{p,\infty}(\torus^d)$ 
\begin{align} \label{eq:general_kappa}
	\norm{(I-P_k)h}_{L^p}\le C2^{-kt}\norm{h}_{B^t_{p,\infty}},
\end{align}
where $C$ is independent of $\omega^*$ and $k$. And we have for all $q\in (1,p]$, $t,r>0$, with $t<r$ that
\begin{align}\label{eq:general_sigma}
	\norm{P_k}_{B^t_{p,\infty}\to B^r_{q,2}}\le C 2^{k(r-t)},
\end{align}
where again, $C$ is independent of $k$. The second inequality can be established as follows: 
From the fact that $N_{t,p,q}(f):=\paren{\sum_{l=0}^\infty 2^{ltq}\|P_lf-P_{l-1}f\|_{L^p}^q}^{1/q}$
(with $P_{-1}:=0$ and change to supremum norm if $q=\infty$) is an equivalent norm on $B^{s}_{p,q}(\torus^d)$ 
(\cite[Thm.~5.1]{dVP:88})  we conclude that
\begin{align*}
	\norm{P_k}_{B^t_{p,\infty}\to B^r_{p,q}}\le C\sup_{f \text{ with } N_{t,p,\infty}(f)\le 1 }N_{r,q,2}(P_k f)
\end{align*}
We have $P_lP_k=P_{\min(k,l)}$ and therefore $P_lP_kf-P_{l-1}P_kf=0$ for $l>k$. From $N_{t,p,\infty}(f) =\sup_{l\in \N}2^{lt}\|P_lf-P_{l-1}f\|_{L^p} \le 1$ we can conclude
$\|P_lf-P_{l-1}f\|_{L^q}\le 2^{-lt}$ by the continuity of the embedding 
$L^p(\torus^d)\hookrightarrow L^q(\torus^d)$. Thus we find 
\begin{align*}
	\norm{P_k}_{B^t_{p,\infty}\to B^r_{p,q}}\le C \paren{\sum_{l=0}^k 2^{2lr}2^{-2lt}}^\frac{1}{2}\le C 2^{k(r-t)}.
\end{align*}
The proof works for both $p=2,p=\infty$ simultaneously, but we distinguish between the two smoothness assumptions. 

\emph{Case 1:}
As $\rho\leq \fdagger/f_0\leq \rho^{-1}$ and $\ln$ restricted to 
$[\rho,\rho^{-1}]$ is infinitely smooth, it follows from the theorem in
\cite{kateb:00} that  $\ln(\fdagger/f_0)\in B^s_{p,\infty}$. As $T^*: B^{s-a}_{p,\infty}(\torus^d)\to B^s_{p,\infty}(\torus^d)$ is an 
isomorphism we have $\pbar\in B^{s-a}_{p,\infty}$. $\Rpen''[\f]=1/\f$ is uniformly bounded in a small neigborhood of $\fdagger\ge \rho$  so we can conclude from \eqref{eq:general_kappa} that 
$\kappa_k\le C2^{-k(s-a)}\norm{\pbar}_{B^{s-a}_{p,\infty}}$. 
It follows from \eqref{eq:general_sigma} that 
\begin{align*}
	\norm{T^{-1}P_k\pbar}_{L^p} \le C\norm{T^{-1}P_k\pbar}_{B_{p,2}^0}
	&\le C\norm{T^{-1}}_{B_{p,2}^a\to B_{p,2}^0}\norm{P_k}_{B_{p,\infty}^{s-a}\to B_{p,2}^a}\norm{\pbar}_{B_{p,\infty}^{s-a}}\\
	&\le C 2^{k(2a-s)} \norm{\pbar}_{B_{p,\infty}^{s-a}},
\end{align*}
so  $\sigma_k\leq\max\{1,C2^{k(2a-s)}\norm{\pbar}_{B_{p,\infty}^{s-a}}\}$. 
Then Proposition \ref{prop_vsc2_proof} and the choice $2^{-k}\sim \tau^{1/(2s)}$ 
show that $\VSC^2(\fdagger,\Phi,\KL(\cdot,\f_0),\Sfun_{\rm sq})$ holds true with
\begin{align*}
\Phi(\tau) &= C \inf_{k\in \mathbb{N}}\left[2^{-k(s-2a)}\sqrt{\tau}+  2^{-k(2s-2a)}\right]
\leq C \tau^{\frac{s-a}{s}}.
\end{align*}
For the last statement we apply \eqref{eq:Grasmairs_rates} with $q=2$ and note 
that $(-\Phi)^*(x)= C(-x)^{(a-s)/a}$ for $x<0$ and $(-\Phi)^*(x)=\infty$ else. Hence, 
$\KL\paren{\fdagger,\itfal{2}}\leq C(\delta^2/\alpha + \alpha^{s/a})$, and the choice 
$\alpha\sim \delta^{\frac{2a}{s +a}}$ leads to \eqref{eq:ME_rate}.

\emph{Case 2:} 
Assumption \eqref{eq:bregman_lower_bound} of Proposition \ref{prop_vsc3_proof} is satisfied with $\mu=2$ due to the inequalities
\begin{subequations}
\begin{align*}
&2\KL(f_1,f_2)\geq \|f_1-f_2\|_{L^1}^2&& \mbox{if  } \|f_1\|_{L^1} = \|f_2\|_{L^1}=1\\
&\paren{\frac{2}{3}\|f_1\|_{L^\infty}+\frac{4}{3}\|f_2\|_{L^\infty}}\KL(f_1,f_2)
\geq \frac{1}{2}\|f_1-f_2\|_{L^2}^2&& \mbox{if  } \|f_1\|_{L^\infty}, \|f_2\|_{L^\infty}<\infty
\end{align*}
\end{subequations}
for $p=\infty$ and $p=2$, respectively (see \cite[Prop. 2.3]{BL:91} and \cite[Lemma 2.6]{HW:16}). 
By $\fdagger\ge \rho$ we have $\f_0\ge \rho^2$, so using the theorem in \cite{kateb:00} and  infinite smoothness of 
$F(x):=1/x$ on $[\rho^2,\infty)$ we obtain 
$1/\fdagger,1/\f_0 \in B^s_{p,\infty}(\torus^d)$. 
It then follows from \cite[Thm 6.6, case 1b]{johnsen:95}  that $\fdagger/f_0\in B^s_{p,\infty}(\torus^d)$.
As $\rho\leq \fdagger/f_0\leq \rho^{-1}$ and $\ln$ restricted to 
$[\rho,\rho^{-1}]$ is infinitely smooth, it follows again from 
\cite{kateb:00} that  $\ln(\fdagger/f_0)\in B^s_{p,\infty}$. 
{\new Thus we have $\ln(\fdagger/f_0)=T^*T\wbar$ for some $\wbar\in\X$ and as} $T^*T: B^{s-2a}_{p,\infty}(\torus^d)\to B^s_{p,\infty}(\torus^d)$ is an 
isomorphism, we obtain $\wbar\in  B^{s-2a}_{p,\infty}(\torus^d)$. By our assumptions we 
have $s-2a-1/p>0$ and hence $\wbar\in L^{\infty}(\torus^d)$ by the standard 
embedding theorem (see \cite[Thm.~4.6.1]{triebel:78}). 
In particular, $\Rpen$ is Fr\'echet-differentiable at $\fdagger-t\wbar$ w.r.t.\ $L^{\infty}(\torus^d)$ 
for {\new $t<\overline{t}:=\rho/\|\wbar\|$ with $f^*_t:=\Rpen'[\fdagger-t\wbar]$ given by 
$\langle f^*_t,h\rangle = \langle \ln(\fdagger-t\wbar)-\ln f_0,h\rangle$ (see \eqref{eq:KL_derivatives}). 
Therefore,   assumption \eqref{eq:taylor_for_subdiff} of Proposition \ref{prop_vsc3_proof} is satisfied
with
\[\omega^*=\frac{\wbar}{\fdagger}.
\]}%
Again using \cite[Thm 6.6, case 1b]{johnsen:95} we obtain 
\[
\omega^* \in B^{s-2a}_{p,\infty}(\torus^d). 
\]
We conclude from \eqref{eq:general_kappa} that 
$\kappa_k\leq  C2^{-k(s-2a)}\norm{\omega^*}_{B^{s-2a}_{p,\infty}}$. 
Moreover, 
it follows from \eqref{eq:general_sigma} that both for $p=2$ and $p=\infty$ we have 
\[
\sigma_k \leq \|(\T^*)^{-1}\|_{B^a_{2,2}\to L^2} \norm{P_k}_{B^{s-2a}_{p,\infty}\to B^{a}_{2,2}} 
\norm{\omega^*}_{B^{s-2a}_{p,\infty}}
\leq  C 2^{-k(s-3a)}\norm{\omega^*}_{B^{s-2a}_{p,\infty}}\,.
\]
Now Proposition \ref{prop_vsc3_proof} and the choice $2^{-k}\sim \tau^{1/(2s-2a)}$ 
show that $\fdagger$ satisfies $\VSC^3(\fdagger,\Phi,\KL(\cdot,\f_0),\Sfun_{\rm sq})$ with $\mu=2$ and 
\begin{align*}
\Phi(\tau) &\leq C \inf_{k\in \mathbb{N}}\left[2^{-k(s-3a)}\sqrt{\tau}+  2^{-k(2s-4a)}\right]
\leq C \tau^{\frac{s-2a}{s-a}}.
\end{align*}
For the last statement we apply Theorem \ref{rate_theorem} with $r= q=\mu=2$ and note 
that $\widetilde{\Phi}=\Phi$, 
$(-\Phi)^*(x)= C(-x)^{(2a-s)/a}$ for $x<0$ and $(-\Phi)^*(x)=\infty$ else. Hence, 
$\KL\paren{\fdagger,\itfal{2}}\leq C(\delta^2/\alpha + \alpha^{s/a}+\beta\alpha^4)$, and the choice 
$\alpha\sim \delta^{\frac{2a}{s +a}}$ leads to \eqref{eq:ME_rate}.
\qed\end{proof}

\begin{rem}
It can be shown in analogy to case 1 that the convergence rate \eqref{eq:ME_rate} also holds 
true for $s\in (0,a)$ if $\fdagger/\f_0\in B^s_{p,\infty}(\torus^d)$. To see this, first note 
in analogy to Proposition \ref{prop_vsc3_proof} that $\VSC^1(\fdagger,\Phi,\Rpen,\Sfun)$
(see \eqref{eq:vsc1}) holds true with $\Phi$ defined 
in \eqref{eq:Phi_inf} if after replacing $\omega^*$ by $\f^*=\ln(\fdagger/\f_0)$ in 
\eqref{eq:proj_for_omegastar} we have $\kappa_k\to 0$ and $\sigma_k$ well-defined and finite 
for all $k\in\mathbb{N}$.

Concerning the optimality of the rate \eqref{eq:ME_rate} we refer e.g.\ to \cite{HW:17}.
The gap $[2a,2a+d/2)$ in the Nikolskii scale $B^s_{2,\infty}(\torus^d)$ in which we are 
not able to derive this rate may be due to technical difficulties. 
In the H\"older-Zygmund scale  $B^s_{\infty,\infty}(\torus^d)$ the only 
gaps are at integer multiples of $a$. In contrast, for quadratic regularization we had gaps only 
at integer multiples of $a$ also in the Nikolskii scale (see Corollary \ref{cor:besov_char}). 
\end{rem}

\section{Numerical results}\label{sec:numerical}

In this section we give some numerical results for the iterated maximum entropy regularization.

\emph{Test problem:} \hspace*{1ex}
We choose $T\colon L^1\left(\torus\right)\to  
L^2(\torus)$ to be the periodic convolution operator $(T\f)(x) := \int_0^1 k(x-y)\f(y)\,dy$ 
with kernel
\[{\new
k(x)= \sum_{j=-\infty}^{\infty} \exp(|x-j|/2) }
= \Big(\sinh\frac{1}{4}\Big)^{-1}
\cosh\frac{2x-2\lfloor x\rfloor-1}{4}, \qquad x\in\R
\]
where $\lfloor x\rfloor:=\max\left\{n\in\Z\colon n\leq x\right\}$. Then integration by parts 
shows that $T=(-\partial_x^2+(1/4)I)^{-1}$, and hence $T$ satisfies the assumptions of Theorem \ref{thm:ME_rates}
with $a=2$.  
We choose $\f_0=1$ and the true solution $\fdagger$ such that $\fdagger-1$ is the standard B-spline $B_5$ 
of order $5$ with 
$\supp(B_5)=[0,1]$ and equidistant knots.  Then we have $\fdagger\in B_{2,\infty}^{5.5}(\torus)$, i.e. $s=5.5$.  
(To see this note that piecewise constant functions belong to $B_{2,\infty}^{0.5}(\torus)$ using the defintion 
of this space via the modulus of continuity.) Hence, according to Theorem \ref{thm:ME_rates} a third 
order variational source condition condition $\VSC^3(\fdagger, A\tau^{3/7},\KL(\cdot,1),\Sfun_{\rm{sq}})$ is satisfied 
for some $A>0$.

\emph{Implementation:}\hspace*{1ex}
The operator $\T$ is discretized by sampling $k$ and $\f$ on an equidistant grid with $480$ points.
Then matrix-vector multiplications with $T=T^*$ can be implemented efficiently by FFT. 
The minimizers $\fal$ and $\itfal{2}$ are computed by the Douglas-Rachford algorithm.
To be consistent with our theory, we consider the constraint set 
$\mathcal{C}:=\{\f\in L^1(\mathbb{T})\colon 0\leq \f\leq 5\mbox{ a.e.}\}$.
We checked that for none of the unconstrained minimizers the bound constraints were active 
such that an explicit implementation of these constraints was not required for our test problem. 

To check the predicted convergence rates with respect to the noise level $\delta$ the regularization parameter 
$\alpha$ was chosen by an a-priori rule of the form {\new $\alpha=c\delta^\sigma$} with an optimal exponent $\sigma>0$ 
and a constant $c$ chosen to minimize the constants for the upper bound given in the figures. 
As we bound the worst case errors in our analysis we tried to approximate the worst case noise. 
Let $G_\delta:=\{\gdagger+\delta\sin(2\pi k\cdot):k\in\N\}$. For each value of $\delta$ we found 
$\gobs\in G_\delta$ such that the reconstruction error gets maximal. This in particular yielded larger 
propagated data errors than discrete white noise. 

\begin{figure}
	\centering
  \includegraphics[width=0.7\textwidth]{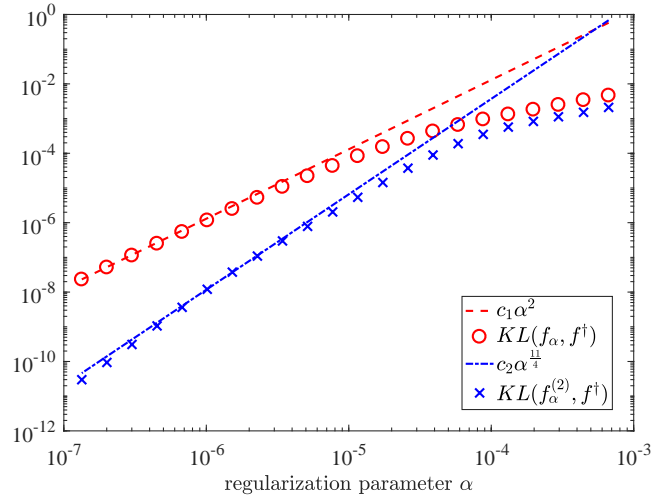} 
  \caption{Predicted and computed approximation error for standard and iterated maximum entropy regularization.}\label{fig:approx}
\end{figure}

\begin{figure}
	\centering
	\includegraphics[width=0.7\textwidth]{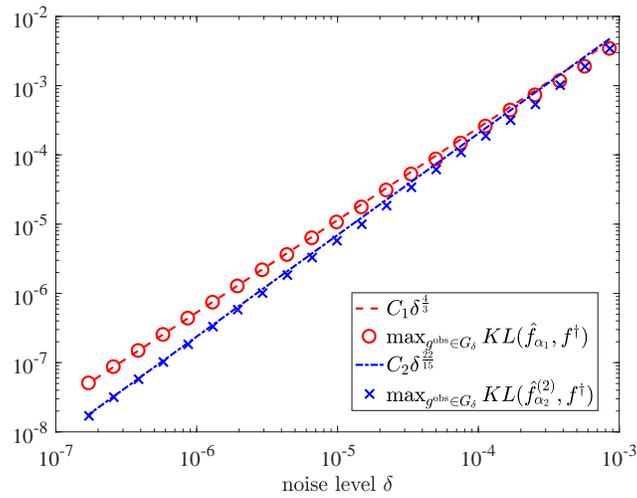}
	\caption{Predicted and computed convergence rates for standard and iterated maximum entropy regularization.}\label{fig:rates}
\end{figure}

\emph{Discussion of the results:} \hspace*{1ex}
Figure \ref{fig:approx} shows the approximation error as a function of $\alpha$, 
i.e.\ $\KL(\f_\alpha,\fdagger)$ where $\f_\alpha$ and  $\f_\alpha^{(2)}$, rsp.,  are the reconstructions for exact data $\gobs=\gdagger$. The two dashed lines indicate the corresponding asymptotic convergence rates predicted by 
our theory, which are in good agreement with the empirical results. 
Note that the saturation effect limits the convergence of the standard
maximum entropy estimator $\f_\alpha=\f_\alpha^{(1)}$  to the maximal rate $\KL(\f_\alpha,\fdagger)=\mathcal{O}(\alpha^{2})$. 
Iterating maximum entropy estimation yields a clear improvement to 
$\KL(\f_{\alpha}^{(2)},\fdagger)=\mathcal{O}(\alpha^{s/a})=\mathcal{O}(\alpha^{11/4})$. 

Figure \ref{fig:rates} displays the convergence rates with respect to the noise level $\delta$ for 
the a-priori choice rule of $\alpha$ described above. Of course, in practice one would rather use some 
a-posteriori stopping rule such as the Lepskii balancing principle, but this is not in the scope 
of this paper. Again, we observe very good agreement of the empirical rate 
$\KL(\fal,\fdagger)=\mathcal{O}(\delta^{4/3})$ with the maximal rate for non-iterated maximum 
entropy regularization, as well as agreement of the rate $\KL(\itfal{2},\fdagger)
= \mathcal{O}(\delta^{2s/(s+a)})=\mathcal{O}(\delta^{22/15})$
of the Bregman iterated estimator $\itfal{2}$ with the rate predicted by Theorem~\ref{thm:ME_rates}.



\section{Discussion and outlook}	
We have shown that variational source conditions can yield convergence rates of arbitrarily high order 
in Hilbert spaces. 
Furthermore we have used this approach to show third order convergence rates in Banach spaces for the first time. 

This naturally leads to the question about arbitrarily high order convergence rates in Banach spaces. 
There are some difficulties that prevented us from going to fourth order convergence rates. 
The approach in Section \ref{sec:third_order} relies on comparison with the convergence rates for the dual variables. As the dual problem to generalized Tikhonov regularization is again some form of generalized Tikhonov regularization, it has finite qualification. Therefore, it does not seem straightforward to get to 
higher orders with this approach. For the approach in Section \ref{sec:Hilbert_spaces} one needs some 
relation between $\breg{\Rpen}(\itfal{2},\fdagger)$ and $\breg{\Rpen}(\fal,\fdagger-\alpha^{q^*-1}\wbar)$, which 
is established in  \eqref{eq:polarization_id} and  \eqref{bound_in_terms_of_itfalk} using the polarization identity. However, this identity only has generalizations in the form of inequalities in Banach spaces.

We hope that the tools provided in this paper will initialize a further development of regularization theory 
in Banach spaces concerning higher order convergence rates. Topics of future research may include 
other regularization methods 
(e.g.\ iterative methods), verifications of higher order variational source conditions for \emph{non-smooth} penalty terms, stochastic noise models, more general data fidelity terms, or nonlinear forward operators. 

\appendix
\section{Duality mappings and an inequality by Xu and Roach}
In this appendix we derive a lower bound on Bregman distances in terms of norm powers 
from more general inequalities by Xu and Roach. 
First recall the following definitions (see e.g.~\cite{Lindenstrauss1979}): 
\begin{defi}\label{defi:smooth_convex}
	The modulus of convexity $\delta_\Y\colon (0,2]\to [0,1]$ of the space $\Y$ is defined by
	\begin{equation*}
		\delta_\Y(\varepsilon):=\inf \{1-\norm{y+\tilde{y}}/2:y,\tilde{y}\in\Y,\norm{y}=\norm{\tilde{y}}=1, \norm{y-\tilde{y}}=\varepsilon\}.
	\end{equation*}
	The modulus of smoothness $\rho_\Y\colon (0,\infty)\to (0,\infty)$ of $\Y$ is defined by
	\begin{equation*}
		\rho_\Y(\tau):= \sup\{(\norm{y+\tilde{y}}+\norm{y-\tilde{y}})/2-1:y,\tilde{y}\in\Y, \norm{y}=1,\norm{\tilde{y}}=\tau\}.
	\end{equation*}
	The space $\Y$ is called \textit{uniformly convex} if $\delta_\Y(\varepsilon)>0$ for every $\varepsilon>0$. It is called \textit{uniformly smooth} 
	if $\lim_{\tau\to 0} \rho_\Y(\tau)/\tau=0.$
	The space $\Y$ is called \textit{$\conv$-convex} (or convex of power type $\conv$) if there exists a constant $K>0$ such that $\delta_\Y(\varepsilon)\ge K\varepsilon^\conv$ for all $\varepsilon>0$. Similarly, it is called $\smoo$-smooth (or smooth of power type $\smoo$) if $\rho_\Y(\tau)\le K\tau^\smoo$ for all $\tau>0$.
\end{defi}

As an example we mention that $L_p$ spaces with $1<p<\infty$ are $\min(p,2)$-smooth and $\max(p,2)$-convex. 
It is known (see \cite{2001Godefroy}) that every Banach space, which is either uniformly smooth or uniformly convex, allows an equivalent norm with respect to which it is $\conv$-convex and $\smoo$-smooth 
with $1<\smoo\le 2\le \conv<\infty$. By \cite[Proposition 1.e.2]{Lindenstrauss1979} we know that 
$\Y^*$ is $\smoo^*$-convex and $\conv^*$-smooth.

Recall that $\Sfun=\frac{1}{q}\norm{\cdot}_\Y^q$ with $q>1$. 
By \cite[Chap.1, Theorem 4.4]{Cioranescu1990}  we have $\partial \Sfun(y)=J_{q,\Y}(y)$, where $J_{q,\Y}$ is the duality mapping given by
\begin{align}\label{eq:defi_duality_map}
	J_{q,\Y}(y)&:=\left\lbrace \omega\in\Y^*: \lsp\omega, y \rsp=\norm{\omega}\norm{y},\norm{\omega}=\norm{y}^{q-1}\right\rbrace.
\end{align}
$J_{q,\Y}$ is $(q-1)$-homogeneous, i.e.\ for all $\lambda\in\R$ we have
\begin{equation}
	J_{q,\Y}(\lambda y)=\sgn(\lambda)|\lambda|^{q-1}J_{q,\Y}(y). \label{homogeneity}
\end{equation}
We assume that $\Y$ is $q$-smooth, $J_{q,\Y}$ is single-valued \cite[Chap.1, Corollary 4.5]{Cioranescu1990}, 
and we can drop superscripts in Bregman distances. 

\begin{lem}[Xu-Roach] \label{xu-roach}
Let $\Y$ be an $\conv$-convex 
Banach space and $\Sfun=\frac{1}{q}\norm{\cdot}_\Y^q$ for some $q>1$. 
Then there exist a constant $\cXR>0$ 
depending only on $q$ and the space $\Y$ such that for all $x,y\in\Y$ we have
	 \begin{align*}
	 	\breg{\Sfun}(x,y) \ge
	 	\begin{cases}
	 		\cXR\max\{\|y\|,\|x-y\|\}^{q-r}\|x-y\|^r &\text{if } q \le r, \\
	 		\cXR\|y\|^{q-r}\|x-y\|^r &\text{if } q \ge r.
	 	\end{cases}
	 \end{align*}
\end{lem}
\begin{proof}
Let $q\le \conv$.	By \cite[Theorem 1]{Xu1991} there exists a constant 
	$C$ depending only on $q$ and $\Y$ such that
	\begin{align*}
		\breg{\Sfun}(x,y)\ge C\int_0^1\frac{t^{r-1}}{2^r}\max\{\|y\|,\|y+t(x-y)\|\}^{q-r}\|x-y\|^r\diff t.
	\end{align*}
	As $\max\{\|y\|,{\new \|y+t(x-y)\|}\}\leq 2\max\{\|y\|,\|x-y\|\}$ for all $t\in [0,1]$ and  
	$q-r\le 0$, we conclude
	\begin{align*}
		\breg{\Sfun}(x,y)\ge 2^{q-r}C\max\{\|y\|,\|x-y\|\}^{q-r}\|x-y\|^r \int_0^1\frac{t^{r-1}}{2^r}\diff t.
	\end{align*}
	This shows the lower bound with $\cXR:=2^{q-r}C\int_0^1\frac{t^{r-1}}{2^r}\diff t>0$ 
	in this case. 
If $r<q$ we have $\max\{\|y\|,\|y+t(x-y)\|\}^{q-r}\ge\|y\|^{q-r}$, 
so that the claim follows as above.
\qed\end{proof}


\begin{acknowledgements}
{\new We would like to thank two anonymous referees for their comments, which helped to improve the paper. }
	Financial support by Deutsche Forschungsgemeinschaft through grant CRC 755, project C09, and RTG 2088 is gratefully acknowledged. 
\end{acknowledgements}

\bibliography{database}   

\end{document}